\documentclass[12pt,twoside]{article}
\usepackage[a4paper,centering,vscale=0.8,hscale=0.8]{geometry}

\usepackage[T1]{fontenc}
\usepackage[latin1]{inputenc}
\usepackage[english]{babel}
\usepackage[babel]{csquotes}

\usepackage{cite}

\usepackage{amssymb}
\usepackage{amsmath}
\usepackage{amsthm}
\usepackage{latexsym}
\usepackage{graphicx}
\usepackage{mathrsfs}
\usepackage{mathrsfs}
\usepackage{bbm}

\usepackage[usenames,dvipsnames]{color}

\theoremstyle{plain}
\newtheorem{thm}{Theorem}[section]

\newtheorem{lem}[thm]{Lemma}

\theoremstyle{definition}
\newtheorem{rmk}[thm]{Remark}

\def\enne{\mathbb{N}}
\def\zeta{\mathbb{Z}}

\def\erre{\mathbb{R}}

\def\H{\mathcal{H}}
\def\V{\mathcal{V}}
\def\W{\mathcal{W}}

\def\eps{\varepsilon}

\def\beq{\begin{equation}}
\def\eeq{\end{equation}}

\def\to{\rightarrow}
\def\wto{\rightharpoonup}
\def\wstarto{\stackrel{*}{\rightharpoonup}}
\def\embed{\hookrightarrow}

\def\norm #1{\left\|#1\right\|}
\def\abs #1{\left|#1\right|}
\def\ip #1#2{\left<#1,#2\right>}

\title{\huge\rm Existence and uniqueness of solutions\\to singular Cahn-Hilliard equations\\
	with nonlinear viscosity terms\\
	and dynamic boundary conditions\footnote{{\bf Acknowledgments.}
	The author is very grateful to Pierluigi Colli for his expert support and 
	fundamental advice. The author is also
	thankful for the warm hospitality and excellent working conditions at the
	Dipartimento di Matematica "F.~Casorati", Universit\`a di Pavia (Italy), where a part of this 
	work was written.}
	\\[.5cm]}

\author{{\large\sc Luca Scarpa}\\
	{\small Department of Mathematics, University College London}\\
	{\small Gower Street, London WC1E 6BT, United Kingdom}\\
	{\small E-mail: \texttt{luca.scarpa.15@ucl.ac.uk}}
	}

\date{}

\begin{document}

\maketitle  

\begin{abstract}
  We prove global existence and uniqueness of solutions to a Cahn-Hilliard
  system with nonlinear viscosity terms and nonlinear dynamic boundary conditions.
  The problem is highly nonlinear, characterized by four nonlinearities and 
  two separate diffusive terms, all acting in the interior of the domain or on its boundary.
  Through a suitable approximation of the problem based on abstract theory 
  of doubly nonlinear evolution equations, existence and uniqueness of solutions
  are proved using compactness and monotonicity arguments.
  The asymptotic behaviour of the solutions as the the diffusion operator on the boundary
  vanishes is also shown.
  \\[.5cm]
\end{abstract}

\pagestyle{myheadings}
\newcommand\testopari{\sc Luca Scarpa}
\newcommand\testodispari{\sc Cahn-Hilliard equations with nonlinear viscosity and D.B.C.}
\markboth{\testodispari}{\testopari}


\thispagestyle{empty}

\section{Introduction}
\setcounter{equation}{0}
\label{intro}

The viscous Cahn-Hilliard equation can be written in its general form as
\[
  \partial_t u - \Delta \mu = 0\,, \qquad
  \mu = \alpha(\partial_t u) - \Delta u + \beta(u) + \pi(u) - g \qquad \text{in }(0,T)\times\Omega\,,
\]
where the unknown $u$ and $\mu$ represent the so-called order parameter and
chemical potential, respectively. Such equation is fundamental in the
phase-separation of a binary alloy, for example, and describes important
qualitative behaviour like the so-called spinodal decomposition:
we refer to the classical works \cite{cahn-hill, novick-cohen, maier-stan1, maier-stan2}
for a physical derivation of the model and some studies on the spinodal decomposition process.
Here, $\Omega$ is smooth bounded domain in $\erre^N$ ($N=2,3$)
with smooth boundary $\Gamma$, and $T>0$ is the final time.
As usual, the term $\beta+\pi$ represents the derivative of a double-well potential, $g$ is a given source
and $\alpha$ is a monotone function acting on $\partial_t u$.
While in the original model $\alpha$ is a linear function, some generalizations
have been proposed where the behaviour of $\alpha$ is of nonlinear type:
see in this direction \cite{gurtin}.

In the present contribution, we study the equation above coupled with 
the homogenous Neumann boundary condition for $\mu$
\[
  \partial_{\bf n}\mu = 0 \qquad\text{in } (0,T)\times\Gamma\,,
\]
which is very natural and ensures the conservation of the mass in the bulk, 
and a second-order doubly nonlinear dynamic boundary condition for $u$
\[
  \alpha_\Gamma(\partial_t u) + \partial_{\bf n}u - \eps\Delta_\Gamma u + \beta_{\Gamma}(u) + \pi_\Gamma(u) = g_\Gamma
  \qquad\text{in } (0,T)\times\Gamma\,.
\]
Here, $\eps>0$ is a fixed constant, $\Delta_\Gamma$ is the usual Laplace-Beltrami
operator on $\Gamma$, $g_\Gamma$ is a prescribed source on the boundary and 
the term $\beta_\Gamma + \pi_\Gamma$ represents the derivative of a double-well potential on the boundary,
which may possibly differ from the one in the interior of the domain $\Omega$.
Similarly, $\alpha_\Gamma$ is a generic monotone function.
Dynamic boundary conditions have been recently proposed by physicists 
in order to take into account also possible interactions with the walls of a confined system:
for a physical motivation of this choice and some
studies on parabolic-type equations with dynamic boundary conditions
we mention the works \cite{fish-spinod, kenz-spinod} and
\cite{gal-DBC, gal-DBC2, gal-grass-ACDBC}.

Cahn-Hilliard equations with dynamic boundary have been widely studied in the last years
in the classical setting in which the viscosity terms depend linearly on the time-derivative
of the order parameter. This framework corresponds in our notation to the choices
$\alpha=a I$ and $\alpha_\Gamma= b I$, with $a,b>0$ given constants
and $I$ the identity on $\erre$.
Let us mention
in this direction the works \cite{mir-zel, col-fuk-eqCH,
colli-fuk-CHmass, col-gil-spr, gil-mir-sch, gil-mir-sch-longtime, col-scar} 
dealing with well-posedness, regularity, long-time behaviour of solutions and asymptotics,
\cite{col-far-hass-gil-spr, col-gil-spr-contr, col-gil-spr-contr2} for some corresponding 
optimal control problems, and \cite{cal-colli, colli-sprek-optACDBC} focused 
specifically on Allen-Cahn equations.

On the other hand, an important area of interest has been equally developed 
on the study of Cahn-Hilliard equations with possibly nonlinear viscosity terms:
the reader can refer to the contributions \cite{mir-sch} for existence-uniqueness and long-time behaviour
under classical homogeneous Neumann conditions, and to
\cite{bcst1} for a detailed thermodynamical derivation of the model
and well-posedness in the case of Dirichlet conditions for the chemical potential. 
Let us also mention the work \cite{mir-zel2} dealing with 
a doubly nonlinear Cahn-Hilliard equation with a different type
of nonlinearity in the viscosity, and the classical contributions
\cite{colli-visin, sch-seg-stef} on a variational approach to 
abstract doubly nonlinear equations.
As the reader may notice, in this case the attention is mainly focused on the presence of a double nonlinearity in the governing equation, 
and, consequently, the prescription on the boundary conditions remains quite broad and classical
(homogeneous Neumann or Dirichlet type).

The aim of this paper is to provide some unifying existence and uniqueness results
for the more general case when both dynamic boundary conditions and 
nonlinear viscosity terms are present in the system. 
From the physical point of view, the presence of dynamic boundary conditions and 
nonlinear viscosity terms is more accurate, and allows for a more
genuine description on the process. On the other side,
from the mathematical perspective, the model is much more difficult
to handle and to study. Indeed,
this specific description gives rise to a system with
$4$ nonlinearities: $\alpha$ and $\alpha_\Gamma$ acting on the time-derivatives
and representing the viscosities, and $\beta$ and $\beta_\Gamma$
acting on the order parameter. Besides the non-triviality of the model,
the presence of several nonlinearities is strongly stimulating and challenging. 
In order to include also
possibly non-smooth potentials in our analysis, the nonlinearities are
assumed to be possibly multivalued (maximal monotone) graphs.

To summarize, we are concerned with the following system
\begin{align}
  \label{eq1}
  \partial_t u - \Delta \mu = 0 \qquad&\text{in } (0,T)\times\Omega\,,\\
  \label{eq2}
  \mu \in \alpha(\partial_t u) - \Delta u + \beta(u) + \pi(u) - g \qquad&\text{in } (0,T)\times\Omega\,,\\
  \label{bound}
  u=v\,, \quad \partial_{\bf n}\mu=0 \qquad&\text{in } (0,T)\times\Gamma\,,\\
  \label{eq3}
  \alpha_\Gamma(\partial_tv) + \partial_{\bf n}u - \eps\Delta_\Gamma v + \beta_\Gamma(v) + \pi_\Gamma(v) \ni g_\Gamma
  \qquad&\text{in } (0,T)\times\Gamma\,,\\
  \label{init}
  u(0)=u_0\,, \quad v(0)=v_0 \qquad&\text{in } \Omega\,.
\end{align}
The paper is organized as follows. In Section~\ref{main_results}
we state the main hypotheses of the work and the main results, commenting
on the different set of assumptions that are in play.
Section~\ref{approx} in entirely focused on the construction of suitable approximated solutions, 
and is based on some abstract results on doubly nonlinear evolution equations.
In Sections~\ref{proof1}, \ref{proof2} and \ref{proof3} we present the proofs of the three
existence results of the paper, while Section~\ref{proof4} contains the proof of the uniqueness
result. Finally, in Section~\ref{proof5} we give a proof of the asymptotic limit as $\eps\to0$, 
recovering in this way a solution to the system corresponding to the case $\eps=0$.


\section{Setting, assumptions and main results}
\setcounter{equation}{0}
\label{main_results}

Throughout the paper, $\Omega\subseteq\erre^N$ ($N=2,3$) is a smooth bounded domain
with smooth boundary $\Gamma$ and $T>0$ is a fixed final time.
We use the notation $Q_t:=(0,t)\times\Omega$ and $\Sigma_t:=(0,t)\times\Gamma$ for every $t\in(0,T]$,
with $Q:=Q_T$ and $\Sigma:=\Sigma_T$.
The outward normal unit vector on $\Gamma$, the tangential gradient and the Laplace-Beltrami 
operator on $\Gamma$ are denoted by ${\bf n}$, $\nabla_\Gamma$ and $\Delta_\Gamma$, respectively.
We shall also use the symbol $\Delta_{\bf n}$ to denote the Laplace
operator with homogeneous Neumann conditions.
Moreover, $\eps$ is a positive fixed number.

We introduce the spaces
\begin{gather*}
  H:=L^2(\Omega)\,, \qquad H_\Gamma:=L^2(\Gamma)\,, \qquad \H:=H \times H_\Gamma\,,\\
  V:=H^1(\Omega)\,, \qquad 
  V_{\Gamma}:=H^1(\Gamma)\,,
  \qquad \V:=\{(x,y)\in V\times V_{\Gamma}: x=y \text{ on } \Gamma\}\,,\\
  W:=H^2(\Omega)\,, \qquad
  W_\Gamma:=H^2(\Gamma)\,, \qquad \W:=\{(x,y)\in W\times W_{\Gamma}: x=y \text{ on } \Gamma\}\,,\\
  W_{\bf n}:=\{x \in W: \partial_{\bf n}x=0 \text{ on } \Gamma\}\,.
\end{gather*}
As usual, we identify $H$ and $H_\Gamma$ with their own duals $H^*$ and $H_\Gamma^*$,
so that $H\embed V^*$ and $H_\Gamma\embed V_{\Gamma}^*$ with the inclusions
given by the inner products of $H$ and $H_\Gamma$, respectively. Moreover, we 
denote all norms and duality pairings by the symbols $\norm{\cdot}$ and  $\ip{\cdot}{\cdot}$, respectively,
with a subscript specifying the spaces in consideration.

For any element $y\in V^*$ we define the mean
\[
  y_\Omega:=\frac1{|\Omega|}\ip{y}{1}\,.
\]
Moreover, recall that a norm on $V$, equivalent to the usual one, is given by
\beq\label{V_eq}
  |x|^2_V:=\norm{\nabla x}_H^2 + |x_\Omega|^2\,, \quad x\in V\,,
\eeq
and that the Laplace operator with Neumann conditions is an isomorphism 
between the null-mean elements in $V$ and the null-mean elements in $V^*$, so that 
it is well defined its inverse
\[
  \mathcal N:\{y\in V^*: \; y_\Omega=0\}\to\{y\in V: \;y_\Omega=0\}\,,
\]
where, for any $y\in V^*$ with $y_\Omega=0$, $\mathcal Ny$ is the unique element in $V$ with null mean such that
\[
  \int_\Omega\nabla\mathcal Ny\cdot\nabla\varphi = \ip{y}{\varphi} \quad\forall\,\varphi\in V\,.
\]

Let us specify the main hypotheses on the data: these will be in order in the whole work and 
will not be recalled explicitly.

We assume that
\[
  \widehat{\alpha}, \widehat\alpha_\Gamma, \widehat{\beta}, \widehat\beta_\Gamma:\erre\to[0,+\infty]
\] 
are proper, convex and lower semicontinuous functions such that
\[
 0 = \widehat\alpha(0)=\widehat\alpha_\Gamma(0)=\widehat\beta(0)=\widehat\beta_\Gamma(0)\,,
\]
and we set
\[
 \alpha:=\partial\widehat\alpha\,, \quad \alpha_\Gamma:=\partial\widehat\alpha_\Gamma\,, \quad
 \beta:=\partial\widehat\beta\,, \quad \beta_\Gamma:=\partial\widehat\beta_\Gamma\,.
\]
Moreover, let
\[
  \pi,\pi_\Gamma:\erre\to\erre \quad\text{Lipschitz continuous}\,, \qquad \pi(0)=\pi_\Gamma(0)=0\,,
\]
and denote by $C_\pi$ and $C_{\pi_\Gamma}$
their respective Lipchitz constants.
We shall always assume that $\alpha_\Gamma$ is coercive and 
that $\beta$ is controlled by $\beta_\Gamma$, i.e.
\begin{gather}
\label{coerc2}
  \exists\, b_1, b_2>0:\quad rs\geq b_1|s|^2-b_2 \quad\forall\,s\in D(\alpha_\Gamma)\,,\quad\forall\,r\in\alpha_\Gamma(s)\,,\\
\label{dom1}
  D(\beta_\Gamma)\subseteq D(\beta) \quad\text{and}\quad
  \exists\,c>0:\abs{\beta^0(s)}\leq c\left(1+\abs{\beta_\Gamma^0(s)}\right) \quad\forall\,s\in D(\beta_\Gamma)\,.
\end{gather}
These hypotheses will be always in order and will not be recalled explicitly throughout the paper.
Note that \eqref{dom1} is typically not new in the literature dealing with Allen-Cahn and Cahn-Hilliard equations with dynamic boundary
conditions: see for example \cite{cal-colli, col-gil-spr}. Moreover, condition \eqref{coerc2} appears also very natural
if we recall that the evolution on the boundary is of order $2$ in space,
hence of Allen-Cahn type.

\medskip

The first existence result that we prove requires additional assumptions on
the graphs $\alpha$ and $\alpha_\Gamma$: in particular, 
and their growth at infinity has to be no more than linear and
also $\alpha$ has to be coercive. On the other side, 
no further hypothesis is made on $\beta$ and $\beta_\Gamma$.
\begin{thm}\label{thm1}
  Suppose that
  \begin{gather}
  \label{g}
  g \in L^2(0,T; H)\,, \qquad
  g_\Gamma \in L^2(0,T; H_\Gamma)\,,\\
  \label{u0}
  u_0 \in V\,,\quad u_{0|\Gamma}\in V_{\Gamma}\,, 
  \qquad \widehat\beta(u_0) \in L^1(\Omega)\,, \qquad \widehat\beta_\Gamma({u_0}_{|\Gamma})\in L^1(\Gamma)\,,\\
  \label{u0_mean}
  (u_0)_\Omega \in \operatorname{Int} D(\beta_\Gamma)\,,\\
  \label{alpha_sub}
  D(\alpha)=D(\alpha_\Gamma)=\erre 
  \quad\text{and}\quad\exists\,L>0: \max\{\abs{\alpha^0(s)}, \abs{\alpha_\Gamma^0(s)}\}\leq L\left(1+|s|\right) \quad\forall\,s\in \erre\,,\\
  \label{coerc1}
  \exists\, a_1, a_2>0:\quad rs\geq a_1|s|^2-a_2 \quad\forall\,s\in D(\alpha)\,,\quad\forall\,r\in\alpha(s)\,.
  \end{gather}
  Then, there exists a septuple $(u,v,\mu,\eta,\xi,\eta_\Gamma,\xi_\Gamma)$ such that
  \begin{gather}
    \label{u}
    u \in L^\infty(0,T; V)\cap H^1(0,T; H)\cap L^2(0,T; W)\,,\\
    \label{v}
    v \in L^\infty(0,T; V_{\Gamma})\cap H^1(0,T; H_\Gamma)\cap L^2(0,T; W_\Gamma)\,,\\
    \label{mu}
    \mu \in L^2(0,T; W_{\bf n})\,,\\
    \label{xi_eta}
    \eta,\xi \in L^2(0,T; H)\,, \qquad
    \eta_\Gamma,\xi_\Gamma \in L^2(0,T; H_\Gamma)\,,\\
    \label{cond}
    v=u_{|\Gamma} \quad\text{a.e.~in } \Sigma\,, \qquad u(0)=u_0\,,\\
    \label{incl}
    \eta\in\alpha(\partial_t u)\,, \;\; \xi \in \beta(u)\quad\text{a.e.~in } Q\,, \qquad
    \eta_\Gamma \in \alpha_\Gamma(\partial_t v)\,, \;\;
     \xi_\Gamma \in\beta_\Gamma(v)\quad\text{a.e.~in } \Sigma
  \end{gather}
  and satisfying
  \begin{gather}
    \label{1}
    \partial_t u - \Delta \mu = 0\,, \\
    \label{2}
    \mu=\eta - \Delta u + \xi + \pi(u) - g\,,\\
    \label{3}
    \eta_\Gamma + \partial_{\bf n} u - \eps\Delta_\Gamma v + \xi_\Gamma + \pi_\Gamma(v) = g_\Gamma\,.
  \end{gather}
\end{thm}
\begin{rmk}
Note that the setting of Theorem~\ref{thm1} includes the classical linear viscosity case, where
$\alpha=a I$ and $\alpha_\Gamma=bI$, for $a,b>0$, and allows
for any choice of the potentials acting on $u$ and $v$, provided that 
the compatibility condition \eqref{dom1} holds.
In particular, we are allowed to consider in the choice of $\widehat\beta+\widehat\pi$ 
and $\widehat\beta_\Gamma+\widehat\pi_\Gamma$ also
logarithmic-type potentials, which are the most 
relevant in terms on thermodynamical consistency of the model, i.e.
\[
  r\mapsto \left((1+r)\ln(1+r)+(1-r)\ln(1-r)\right) - cr^2\,, \qquad r\in(-1,1)\,, \qquad c>0\,.
\]
\end{rmk}

In the next existence result, we show how to remove the coercivity hypothesis \eqref{coerc1} on $\alpha$
by requiring stronger assumptions on the data.
Again, no further restrictions
are assumed on $\beta$ and $\beta_\Gamma$.
Moreover, we stress also that if $\alpha$ is coercive, then the further hypotheses
on the data ensure additional regularities on the solutions. 
\begin{thm}\label{thm2}
  Assume conditions \eqref{u0_mean}--\eqref{alpha_sub} and
  \begin{gather}
  \label{g'}
  g \in L^2(0,T; H)\cap H^1(0,T; V^*)\,, \qquad
  g_\Gamma \in L^2(0,T; H_{\Gamma}) \cap H^1(0,T; V_{\Gamma}^*)\,,\\
  \label{g'_bis}
  g(0) \in H\,, \qquad g_\Gamma(0) \in H_\Gamma\,,\\
  \label{u0'}
  u_0 \in W\,, \qquad u_{0|\Gamma} \in W_\Gamma\,,\\
  \label{u0'_bis}
  \exists\,y_0\in H:\;y_0\in\beta(u_0)\quad\text{a.e.~in } \Omega\,, \qquad
  \exists\,y_{0\Gamma}\in H_\Gamma:\;y_{0\Gamma}\in\beta_\Gamma(u_{0|\Gamma})\quad\text{a.e.~in } \Gamma\,,\\
  \label{u0'_ter}
    \exists\,\delta_0>0:\quad 
    \{-\Delta u_0+\beta_\delta(u_0)+(I-\delta\Delta_{\bf n})^{-1}g(0)\}_{\delta\in(0,\delta_0)} \quad\text{is bounded in } V\,.
  \end{gather}
  Then there exists a septuple $(u,v,\mu,\eta,\xi,\eta_\Gamma,\xi_\Gamma)$ such that
  \begin{gather}
  \label{u'}
  u \in W^{1,\infty}(0,T; V^*)\cap H^1(0,T; V)\cap L^\infty(0,T; W)\,,\\
  \label{v'}
  v \in W^{1,\infty}(0,T; H_\Gamma)\cap H^1(0,T; V_\Gamma)\cap L^\infty(0,T; W_\Gamma)\,,\\
  \label{mu'}
  \mu \in L^\infty(0,T; V)\cap L^2(0,T; W_{\bf n}\cap H^3(\Omega))\,,\\
  \label{xi_eta'}
  \eta\,,\xi \in L^\infty(0,T; H)\,, \qquad \eta_\Gamma\,,\xi_\Gamma \in L^\infty(0,T; H_\Gamma)
  \end{gather}
  and satisfying conditions \eqref{cond}--\eqref{3}. Moreover, if \eqref{coerc1} holds, then the same conclusion
  is true without the assumption \eqref{u0'_ter}, and additionally $u \in W^{1,\infty}(0,T; H)$
  and $\mu \in L^\infty(0,T; W_{\bf n})$.
\end{thm}

\begin{rmk}
  Note that hypothesis \eqref{u0'_ter} clearly holds if $u_0\in H^3(\Omega)$, $g(0)\in V$ and
  the family $\{\beta_\delta(u_0)\}_{\delta \in(0,\delta_0)}$ is bounded in $V$. This condition is not new in literature:
  see for example the work \cite[pp.~977--978]{col-gil-spr} for sufficient conditions.
\end{rmk}
\begin{rmk}
  The setting of Theorem~\ref{thm2} allows to include in our analysis also 
  the cases where $\alpha=\operatorname{sign}$ for example, or
  $\alpha(r)=r_+$, $r\in\erre$. Again, no further assumption are made on $\beta$
  or $\beta_\Gamma$, so that logarithmic-type potentials are included.
\end{rmk}

The third existence result that we present allows to remove the linear growth condition on
$\alpha$ and $\alpha_\Gamma$, but requires in turn a polynomial control 
of the growth of $\beta$ and $\beta_\Gamma$.
Here, 
the inclusions with respect to the operators $\alpha$ and $\alpha_\Gamma$
are satisfied in a weak sense.
To this end, we shall introduce the operators 
$\alpha_w:V\to2^{V^*}$ and $\alpha_{\Gamma w}:V_\Gamma\to 2^{V_\Gamma^*}$
as
\begin{align*}
  \alpha_w(x)&:=\left\{y\in V^*:\;\int_\Omega\widehat\alpha(x) + \ip{y}{\varphi-x}_V\leq\int_\Omega\widehat{\alpha}(\varphi)
  \quad\forall\,\varphi\in V\right\}\,, \qquad x\in V\,,\\
  \alpha_{\Gamma w}(x_\Gamma)&:=\left\{y_\Gamma\in V_\Gamma^*:
  \;\int_\Gamma\widehat\alpha_\Gamma(x_\Gamma) + \ip{y_\Gamma}{\psi-x_\Gamma}_{V_\Gamma}\leq\int_\Gamma\widehat\alpha_\Gamma(\psi)
  \quad\forall\,\psi\in V_\Gamma\right\}\,, \qquad x_\Gamma\in V_\Gamma\,,
\end{align*}
which are clearly the subdifferentials of the proper, convex and l.s.c.~functions induced by $\widehat\alpha$
and $\widehat\alpha_\Gamma$ on $V$ and $V_\Gamma$, respectively. Similarly,
we shall introduce the (maximal monotone) operator $\widetilde\alpha_w:\V\to2^{\V^*}$ as
\[\begin{split}
  \widetilde\alpha_w(x,x_\Gamma):=&\left\{y\in \V^*:\;\int_\Omega\widehat\alpha(x)
  +\int_\Gamma\widehat\alpha_\Gamma(x_\Gamma)
   + \ip{y}{(\varphi,\psi)-(x,x_\Gamma)}_\V\right.\\
   &\qquad\qquad\quad\left.\leq\int_\Omega\widehat{\alpha}(\varphi)
   +\int_\Gamma\widehat\alpha_\Gamma(\psi)
  \quad\forall\,(\varphi,\psi)\in \V\right\}\,, \qquad (x,x_\Gamma)\in \V\,.
\end{split}\]
Note that $\alpha_w(x)+\alpha_{\Gamma w}(x_\Gamma)\subseteq\widetilde\alpha_w(x,x_\Gamma)$ for every $(x,x_\Gamma)\in \V$,
but equality may not hold in general as $\V^*$ is strictly larger than $(V\times V_\Gamma)^*$.
This will result in a weaker variational formulation for both the evolution equation itself
and for the inclusions with respect to the nonlinear operators acting in the viscosity terms.

\begin{thm}\label{thm3}
  Assume conditions \eqref{u0_mean} and \eqref{g'}--\eqref{u0'_ter}. If
  \beq\label{ip_0dom}
    0 \in \operatorname{Int}\left(D(\alpha)\cap D(\alpha_\Gamma)\right)
  \eeq
  and
  \begin{align}
    \label{ip_beta}
    \exists\,c_1,c_2>0&: \quad |r|\leq c_1|s|^5 + c_2 \quad\forall\,s\in D(\beta)\,,\quad\forall\,r\in\beta(s)\\
    \label{ip_beta_g}
    \exists\,p\geq 5,\;d_1,d_2>0&:\quad |r| \leq d_1 |s|^{p} + d_2 \quad\forall\,s\in D(\beta_\Gamma)\,, \quad
    \forall\,r\in \beta_\Gamma(s)\,,
  \end{align}
  then there exists a septuple $(u,v,\mu,\eta,\xi,\eta_\Gamma,\xi_\Gamma)$ such that
  \begin{gather}
    \label{uv_w}
      u \in W^{1,\infty}(0,T; V^*)\cap H^1(0,T; V)\,,\qquad
      v \in W^{1,\infty}(0,T; H_\Gamma)\cap H^1(0,T; V_\Gamma)\,,\\
    \label{mu_w}
     \mu \in L^\infty(0,T; V)\cap L^2(0,T; W_{\bf n}\cap H^3(\Omega))\,,\\
     \label{eta_w}
     \eta_w \in L^\infty(0,T; \V^*)\,, \qquad
     \eta_w\in\widetilde\alpha_w(\partial_t u, \partial_t v) \quad\text{a.e.~in } (0,T)\,,\\
   \label{xi_w}
     \xi \in L^\infty(0,T; L^{6/5}(\Omega))\,, \qquad
     \xi_\Gamma \in L^\infty(0,T; L^q(\Gamma))\quad\forall\,q\in[1,+\infty)\,,\\
    \label{incl_xi_w}
    \xi\in\beta(u) \quad\text{a.e.~in } Q\,,
     \qquad \xi_\Gamma \in \beta_\Gamma(v) \quad\text{a.e.~in } \Sigma\,,
  \end{gather}
  satisfying conditions \eqref{cond}, \eqref{1} and
  \beq
    \label{eq_var}
    \begin{split}
    \int_\Omega\mu(t)\varphi=\ip{\eta_w(t)}{(\varphi,\psi)}_\V &+ \int_\Omega\nabla u(t)\cdot\nabla\varphi + 
    \int_\Omega\left(\xi(t)+\pi(u(t))-g(t)\right)\varphi\\
    &+ \eps\int_\Gamma\nabla_\Gamma v(t)\cdot\nabla_\Gamma\psi
    +\int_\Gamma(\xi_\Gamma(t)+\pi_\Gamma(v(t))-g_\Gamma(t))\psi
    \end{split}
  \eeq
  for every $(\varphi,\psi)\in \V$ and a.e.~$t\in(0,T)$.
  Moreover, if
  \beq
    \label{ip_beta'}
    \exists\,c_1,c_2>0: \quad |r|\leq c_1|s|^3 + c_2 \quad\forall\,s\in D(\beta)\,,\quad\forall\,r\in\beta(s)\,,
  \eeq
  then $\xi \in L^\infty(0,T; H)$. Furthermore, if \eqref{coerc1} holds, then the same conclusions
  are true also without the assumption \eqref{u0'_ter}, and additionally $u \in W^{1,\infty}(0,T; H)$
  and $\mu \in L^\infty(0,T; W_{\bf n})$.
\end{thm}
\begin{rmk}
  The setting of Theorem~\ref{thm3} allows $\alpha$ and $\alpha_\Gamma$ to be superlinear at infinity,
  but in turn requires polynomial growth for $\beta$ and $\beta_\Gamma$.
  In this setting, note that we can include the classical choice
  \[
  r\mapsto \frac14(r^2-1)^2\,, \qquad r\in\erre\,,
  \]
  for $\widehat\beta+\widehat\pi$, and any generic polynomial double-well potential for
  $\widehat\beta_\Gamma+\widehat\pi_\Gamma$. These may be seen, as usual, as
  suitable approximation of the more relevant logarithmic potentials.
\end{rmk}
\begin{rmk}
  Let us stress that the hypothesis \eqref{ip_0dom} is the direct generalization of \eqref{u0_mean}.
  Indeed, it is readily seen from \eqref{1} that $(u)_\Omega$ is constantly equal to $(u_0)_\Omega$,
  as well as $(\partial_t u)_\Omega=0$ at any time. Consequently, taking into account that 
  $\alpha$ and $\alpha_\Gamma$ are acting on the time derivatives of the solutions, 
  the hypotheses \eqref{u0_mean} and \eqref{ip_0dom} clearly possess the same structure.
\end{rmk}
\begin{rmk}
  Let us comment on \eqref{eq_var}, which is the natural variational formulation in the dual space $\V^*$
  of the couple of equations \eqref{eq2} and \eqref{eq3}.
  Note that since $N\in\{2,3\}$, we have the continuous inclusions $V\embed L^6(\Omega)$
  and $V_\Gamma\embed L^q(\Gamma)$ for every $q\in[1,+\infty)$. Hence, it is clear that
  $L^{6/5}(\Omega)\embed V^*$ and $L^{q'}(\Gamma)\embed V_\Gamma^*$ for every $q'\in(1,+\infty]$.
  For these reasons, we have in particular
  that $\xi \in L^\infty(0,T; V^*)$ and $\xi_\Gamma \in L^\infty(0,T;V_\Gamma^*)$, so that the dualities 
  \[
    \int_\Omega \xi\varphi \qquad\text{and}\qquad \int_\Gamma\xi_\Gamma\psi
  \]
  in the variational formulation \eqref{eq_var} make sense by the classical H\"older inequality, 
  and must be read as $\ip{\xi}{\varphi}_V$
  and $\ip{\xi_\Gamma}{\psi}_{V_\Gamma}$, respectively. 
\end{rmk}

We turn now to uniqueness of solutions.
According to different smoothness or growth assumptions on the potentials,
uniqueness in proved
both in the class of solutions given by Theorem~\ref{thm2} and
in the largest class of
Theorem~\ref{thm1}.
\begin{thm}
  \label{thm4}
  Assume that $\beta$ and $\beta_\Gamma$ are single-valued, and that 
  \begin{gather}
  \label{ip_uniq}
    F:=\widehat\beta + \int_0^\cdot\pi(s)\,ds \in C^{2,1}_{loc}(\erre)\,, \qquad
    F_\Gamma:=\widehat\beta_\Gamma + \int_0^\cdot\pi_\Gamma(s)\,ds \in C^{2,1}_{loc}(\erre)\,,\\
  \label{ip_uniq'}
    \exists\,\widetilde{b_1}>0:\quad (s_1-s_2)(r_1-r_2)\geq\widetilde{b_1}|r_1-r_2|^2 \quad\forall\,(r_i,s_i)\in\alpha_\Gamma\,, \;i=1,2\,.
  \end{gather}
  Then, there is a unique septuple $(u,v,\mu,\eta,\xi,\eta_\Gamma,\xi_\Gamma)$ 
  satisfying \eqref{u'}--\eqref{xi_eta'} and \eqref{cond}--\eqref{3}.
  Furthermore, if additionally
  \begin{align}
    \label{ip_uniq''}
    \exists\,M>0:&\quad F'''(r)\leq M\left(1+|r|^3\right) \quad\text{for a.e.~}r\in\erre\,,\\
    \label{ip_uniq'''}
    \exists\,M,q>0:&\quad F_\Gamma'''(r)\leq M\left(1+|r|^q\right) \quad\text{for a.e.~}r\in\erre\,,
  \end{align}
  there exists a unique septuple $(u,v,\mu,\eta,\xi, \eta_\Gamma, \xi_\Gamma)$ 
  satisfying \eqref{uv_w}--\eqref{eq_var}, \eqref{cond} and \eqref{1}.
\end{thm}

Finally, the last result that we present investigates the asymptotic behaviour 
of the solutions with respect to $\eps$, and provides a further existence result
for the problem with $\eps=0$. For sake of brevity, we only consider the case of
the linearity assumption \eqref{alpha_sub} on the growth of $\alpha$ and $\alpha_\Gamma$, 
and provide different asymptotic convergence of the solutions
depending on whether the coercivity assumption \eqref{coerc1} is in order.
In this direction, we need to introduce a weak 
formulation of the operator $\beta_\Gamma$ induced on the space $H^{1/2}(\Gamma)$.
Namely, we define $\beta_{\Gamma w}:H^{1/2}(\Gamma)\to 2^{H^{-1/2}(\Gamma)}$ as
the maximal monotone operator
\[
  \beta_{\Gamma w}(x):=\left\{y \in H^{-1/2}(\Gamma):
  \int_\Gamma\widehat\beta_\Gamma(x) + \ip{y}{w-x}_{H^{1/2}(\Gamma)} \leq \int_\Gamma\widehat\beta_\Gamma(w)
  \quad\forall\,w\in H^{1/2}(\Gamma)\right\}\,.
\]
\begin{thm}
  \label{thm5}
  Assume conditions \eqref{g}, \eqref{u0_mean}--\eqref{coerc1} and let
  \[
  u_0\in V\,, \qquad \widehat\beta(u_0)\in L^1(\Omega)\,, \qquad \widehat\beta_\Gamma(u_{0|\Gamma})\in L^1(\Gamma)\,.
  \]
  Let $(u_0^\eps)_\eps\subseteq V$ be any family such that 
  $u_0^\eps$ satisfies \eqref{u0} for every $\eps>0$,
  \begin{gather}
  u_0^\eps\to u_0 \quad\text{in V}\,, \qquad
  \eps^{1/2}u_{0|\Gamma}^\eps \to 0 \quad\text{in } V_\Gamma \qquad\text{as } \eps\searrow0\,,\\
  \label{est_init}
  \eps\norm{u_{0|\Gamma}^\eps}^2_{V_\Gamma}+
  \norm{\widehat\beta(u_0^\eps)}_{L^1(\Omega)}+
  \norm{\widehat\beta_\Gamma(u_{0|\Gamma}^\eps)}_{L^1(\Gamma)}\leq c \qquad\forall\,\eps>0
  \end{gather}
  for a positive constant $c$, and let 
  $(u_\eps,v_\eps,\mu_\eps,\eta_\eps,\xi_\eps,\eta_{\Gamma\eps},\xi_{\Gamma\eps})$
  be the solutions given by Theorem~\ref{thm1} satisfying conditions \eqref{u}--\eqref{3} with initial datum $u_0^\eps$.
  Then, there exists a sequence $(\eps_n)_n$ with $\eps_n\to0$ as $n\to\infty$
  and a septuple $(u,v,\mu,\eta,\xi,\eta_\Gamma,\xi_\Gamma)$ with
  \begin{gather}
  \label{u_eps0}
  u \in L^\infty(0,T; V)\cap H^1(0,T; H)\,, \quad \Delta u \in L^2(0,T; H)\,,\\
  v \in L^\infty(0,T; H^{1/2}(\Gamma))\cap H^1(0,T; H_\Gamma)\,,\qquad
   \mu \in L^2(0,T; W_{\bf n})\,,\\
   \label{xi_eta_eps0}
   \eta, \xi \in L^2(0,T; H)\,, \quad \eta_\Gamma \in L^2(0,T; H_\Gamma)\,, \quad
   \xi_\Gamma \in L^2(0,T; H^{-1/2}(\Gamma))\,,\\
   \eta\in\alpha(\partial_t u)\,, \quad \xi\in\beta(u) \quad\text{a.e.~in } Q\,, \qquad
   \eta_\Gamma\in\alpha_\Gamma(\partial_t v)\quad\text{a.e.~in } \Sigma\,,\\
   \xi_\Gamma \in \beta_w(v)\quad\text{a.e.~in } (0,T)\,,
  \end{gather}
  satisfying \eqref{cond}, \eqref{1}--\eqref{2} and
  \[
    \eta_\Gamma + \partial_{\bf n} u + \xi_\Gamma + \pi_\Gamma(v) = g_\Gamma\,,
  \]
  and such that, as $n\to\infty$,
  \begin{gather*}
  u_{\eps_n} \to u \quad\text{in } C^0([0,T]; H)\,, \quad
  u_{\eps_n} \wto u \quad\text{in } H^1(0,T; H)\,, \quad
  u_{\eps_n} \wstarto u \quad\text{in } L^\infty(0,T; V)\,,\\
  v_{\eps_n} \to v \quad\text{in } C^0([0,T]; H_\Gamma)\,, \quad
  v_{\eps_n} \wto v \quad\text{in } H^1(0,T; H_\Gamma)\,, \quad
  v_{\eps_n} \wstarto v \quad\text{in } L^\infty(0,T; H^{1/2}(\Gamma))\,,\\
  \mu_{\eps_n}\wto\mu \quad\text{in } L^2(0,T; W_{\bf n})\,,\\
  \eta_{\eps_n}\wto\eta \quad\text{in } L^2(0,T; H)\,,\qquad
  \xi_{\eps_n}\wto\xi \quad\text{in } L^2(0,T; H)\,,\\
  \eta_{\Gamma\eps_n}\wto\eta_\Gamma \quad\text{in } L^2(0,T; H_\Gamma)\,,\qquad
  \xi_{\Gamma\eps_n}\wto\xi_\Gamma \quad\text{in } L^2(0,T; H^{-1/2}(\Gamma))\,,\\
  \eps v_{\eps_n}\to 0 \quad\text{in } L^\infty(0,T; V_\Gamma)\,.
  \end{gather*}
  Furthermore, if also hypotheses \eqref{g'}--\eqref{u0'_ter} hold and $(\eps u_{0|\Gamma}^\eps)_\eps$
  is bounded in $W_\Gamma$, then the same conclusion is true
  without the coercivity assumption \eqref{coerc1}, and we also have
  \begin{gather*}
  u \in W^\infty(0,T; V^*)\cap H^1(0,T; V)\,, \quad \Delta u \in L^\infty(0,T; H)\,,\\
  v \in W^{1,\infty}(0,T; H_\Gamma)\cap H^1(0,T; H^{1/2}(\Gamma))\,,\qquad
   \mu \in L^\infty(0,T; V)\cap L^2(0,T; W_{\bf n}\cap H^3(\Omega))\,,\\
   \eta, \xi \in L^\infty(0,T; H)\,, \quad \eta_\Gamma \in L^\infty(0,T; H_\Gamma)\,, \quad
   \xi_\Gamma \in L^\infty(0,T; H^{-1/2}(\Gamma))\,,
  \end{gather*}
  and
  \begin{gather*}
    u_{\eps_n} \wstarto u \quad\text{in } W^{1,\infty}(0,T; V^*)\,, \qquad
  u_{\eps_n} \wto u \quad\text{in } H^1(0,T; V)\,,\\
  v_{\eps_n} \wstarto v \quad\text{in } W^{1,\infty}(0,T; H_\Gamma)\,, \qquad
  v_{\eps_n} \wto v \quad\text{in } H^1(0,T; H^{1/2}(\Gamma))\,,\\
  \mu_{\eps_n}\wstarto\mu \quad\text{in } L^\infty(0,T; V)\,,\qquad
  \mu_\eps\wto \mu \quad\text{in } L^2(0,T; W_{\bf n}\cap H^3(\Omega))\,,\\
  \eta_{\eps_n}\wstarto\eta \quad\text{in } L^\infty(0,T; H)\,,\qquad
  \xi_{\eps_n}\wstarto\xi \quad\text{in } L^\infty(0,T; H)\,,\\
  \eta_{\Gamma\eps_n}\wstarto\eta_\Gamma \quad\text{in } L^\infty(0,T; H_\Gamma)\,,\qquad
  \xi_{\Gamma\eps_n}\wstarto\xi_\Gamma \quad\text{in } L^\infty(0,T; H^{-1/2}(\Gamma))\,,\\
  \eps v_{\eps_n}\to 0 \quad\text{in } H^1(0,T; V_\Gamma)\,.
  \end{gather*}
\end{thm}

\begin{rmk}
  Let us comment on the existence of an approximating family $(u_0^\eps)_\eps$.
  If the initial datum $u_0\in V$ satisfies $a\leq u_0 \leq b$ almost everywhere in $\Omega$
  for certain $a,b\in\erre$ such that $[a,b]\subseteq\operatorname{Int}D(\beta_\Gamma)$,
  then a possible approximating sequence $(u_0^\eps)_\eps$ always exists. Indeed, we can set,
  for every $\eps>0$, $u_0^\eps$ as the unique solution to the elliptic problem
  \[
    \begin{cases}
    u_0^\eps - \eps^{1/2}\Delta u_0^\eps = u_0 \quad&\text{in } \Omega\,,\\
    \partial_{\bf n} u_0^\eps =0 \quad&\text{in } \Gamma\,.
    \end{cases}
  \]
  Such problem is well-posed by the classical theory on bilinear forms
  and admits a unique solution $u_0^\eps \in W_{\bf n}\cap H^3(\Omega)$.
  Testing by $u_0^\eps$ and using the Young inequality one has
  \[
  \frac12\norm{u_0^\eps}_H^2 + \eps^{1/2}\norm{\nabla u_0^\eps}_H^2 
  \leq \frac12\norm{u_0}_H^2 \,,
  \]
  while testing the first equation by $-\Delta u_0^\eps$ and integrating by parts yields
  \[
  \frac12\norm{\nabla u_0^\eps}^2_H + \eps^{1/2}\norm{\Delta u_0^\eps}_H^2
  \leq \frac12\norm{\nabla u_0}^2_H\,.
  \]
  We infer that (along a subsequence)
  \[
  u_0^\eps \wto u_0 \quad\text{in } V\,, \qquad \norm{u_0^\eps}_V\leq\norm{u_0}_V\quad\forall\,\eps>0\,,
  \]
  so that $u_0^\eps\to u_0$ in $V$, hence also $\eps^{1/2}\Delta u_0^\eps\to 0$ in $V$.
  We deduce then that $\eps^{1/2}u_\eps \to 0$ in $H^3(\Omega)$, which implies in particular
  that $\eps^{1/2}u_{0|\Gamma}^\eps\to 0$ in $V_\Gamma$. Furthermore, by 
  the maximum principle we have $a\leq u_0^\eps\leq b$ a.e.~in $\Omega$, hence also
  $a\leq u_{0|\Gamma}^\eps\leq b$ a.e.~in $\Sigma$, and we can conclude
  recalling that $D(\beta_\Gamma)\subseteq D(\beta)$ and the fact that 
  every proper, convex and lower semicontinuous function is continuous in the interior
  of its domain.
\end{rmk}


\section{The approximated problem}
\setcounter{equation}{0}
\label{approx}

In this section we approximate the problem \eqref{eq1}--\eqref{init} and we precise the exact regularities of
the approximated solutions, depending on the assumptions on the data.
Note that throughout this section $\eps>0$ is fixed, so that we shall omit any 
specific notation for the dependence on~$\eps$.
\smallskip

For any $\lambda>0$, let $\beta_\lambda$ and $\beta_{\Gamma\lambda}$ be the Yosida approximations
of the graphs $\beta$ and $\beta_\Gamma$ with approximating parameters $\lambda$ and $c\lambda$, respectively,
where $c$ is the same as in \eqref{dom1}:
the reason why we choose this specific approximation will be 
clarified in Section~\ref{third} below. 
Similarly, let $\alpha_\lambda$ and $\alpha_{\Gamma\lambda}$ denote
the Yosida approximations of $\alpha$ and $\alpha_\Gamma$, respectively, with parameter~$\lambda$.
Furthermore, let $(g_\lambda)_\lambda$ and $(g_{\Gamma\lambda})_\lambda$ be two approximating
sequences of $g$ and $g_\Gamma$, respectively, such that
\begin{gather*}
  (g_\lambda)_\lambda \subseteq L^2(0,T; V)\cap H^1(0,T; V^*)\,, \qquad
  (g_{\Gamma\lambda})_\lambda \subseteq L^2(0,T; V_\Gamma)\cap H^1(0,T; V_\Gamma^*)\,,\\
  g_\lambda\to g \quad\text{in } L^2(0,T; H)\,, \qquad
   g_{\Gamma\lambda}\to g_\Gamma \quad\text{in } L^2(0,T; H_\Gamma)\,.
\end{gather*}
It will be implicitly intended that the convergences hold also in the spaces 
$H^1(0,T; V^*)$ and $H^1(0,T; V_\Gamma^*)$ whenever \eqref{g'} is in order.
For example, we can define $g:=(I-\lambda\Delta_{\bf n})^{-1}g$ and 
$g_\Gamma:=(I-\lambda\Delta_\Gamma)^{-1}g_\Gamma$, i.e.~as the solutions to the following 
elliptic problems:
\[ 
  \begin{cases}
    g_\lambda - \lambda\Delta g_\lambda = g \quad&\text{in } \Omega\,,\\
    \partial_{\bf n} g_\lambda = 0 \quad&\text{in } \Gamma\,,
  \end{cases}
  \qquad
  g_{\Gamma\lambda} - \lambda\Delta_\Gamma g_{\Gamma\lambda} = g_\Gamma \quad\text{in } \Gamma\,.
\]

The idea is to consider the regularized system given by
\begin{align*}
  \partial_t u_\lambda + \lambda\mu_\lambda - \Delta\mu_\lambda = 0 \qquad&\text{in } Q\,,\\
  \mu_\lambda = \lambda \partial_t u_\lambda + \alpha_\lambda(\partial_t u_\lambda) +\lambda u_\lambda - \Delta u_\lambda
  +\beta_\lambda(u_\lambda) + T_\lambda\pi(u_\lambda) - g_\lambda \qquad&\text{in } Q\,,\\
  u_\lambda=v_\lambda\,, \quad \partial_{\bf n}\mu_\lambda=0 \qquad&\text{in } \Sigma\,,\\
  \lambda\partial_t v_\lambda + \alpha_{\Gamma\lambda}(\partial_tv_\lambda) + \partial_{\bf n}u_\lambda
  - \eps\Delta_\Gamma v_\lambda
  +\beta_{\Gamma\lambda}(v_\lambda) + T_\lambda\pi_\Gamma(v_\lambda) = g_{\Gamma\lambda} \qquad&\text{in } \Sigma\,,\\
  u_\lambda(0)=u_0 \qquad&\text{in } \Omega\,,
\end{align*}
where $T_\lambda:\erre\to\erre$ is the usual truncation operator at level $\frac1\lambda$ defined by
\[
  T_\lambda(r):=\max\left\{-\frac1\lambda, \min\left\{\frac1\lambda, r\right\}\right\}\,, \quad r\in\erre\,.
\]

In order to show that such regularized problem is well-posed, we use an abstract result
on doubly nonlinear evolution equations on the product space $\H$.
To this end, we introduce the operator
\[
  G_\lambda:H\to H\,, \qquad G_\lambda x:=\lambda x-\Delta x\,, \quad x\in D(G_\lambda):=W_{\bf n}\,,
\]
which is maximal monotone and invertible on $H$ with $G_\lambda^{-1}:H\to W_{\bf n}$; 
in particular, the first equation together with the boundary 
condition for $\mu_\lambda$ can be written as $\mu_\lambda=-G_\lambda^{-1}(\partial_t u_\lambda)$.
Hence, it is natural to define
\begin{align*}
  A_\lambda:\H\to\H\,, \qquad &A_\lambda(x,y):=
  (\lambda x + \alpha_\lambda(x) + G_\lambda^{-1}(x), \lambda y + \alpha_{\Gamma\lambda}(y))\,,\\
  B_\lambda:\H\to\H\,, \qquad 
  &B_\lambda(x,y):=(\lambda x-\Delta x + \beta_\lambda(x), \partial_{\bf n}x -\eps\Delta_\Gamma y + 
  \beta_{\Gamma\lambda}(y))\,,
\end{align*}
where
\[
  D(A_\lambda):=\H\,, \qquad
  D(B_\lambda):=\W\,.
\]
Taking into account the definition of $A_\lambda$ and $B_\lambda$,
the entire approximated system can be formulated as a doubly nonlinear 
evolution equation in the variable $(u_\lambda, v_\lambda)$ on the product space $\H$
in the following compact form:
\[
  A_\lambda\partial_t(u_\lambda,v_\lambda) + B_\lambda(u_\lambda,v_\lambda)=
  (g_\lambda,g_{\Gamma\lambda}) - (T_\lambda\pi(u_\lambda), T_\lambda\pi_\Gamma(v_\lambda))\,, \qquad
  (u_\lambda, v_\lambda)(0) = (u_0, u_{0|\Gamma})\,.
\]
We collect some useful properties of the operators $A_\lambda$ and $B_\lambda$ in the following lemma.
\begin{lem}
  \label{prop}
  The operators $A_\lambda$ and $B_\lambda$ are maximal monotone on $\H$
  and $D(B_\lambda)\subseteq \V$. Moreover, the following conditions hold:
  \begin{align*}
    (i)&\qquad \forall\,(x,y)\in\H \quad \left(A_\lambda(x,y), (x,y)\right)_\H\geq\lambda\norm{(x,y)}_\H^2\,,\\
    (ii)&\qquad \exists\,k_\lambda>0:\quad\forall\,(x,y)\in\H 
    \quad \norm{A_\lambda(x,y)}_\H\leq k_\lambda\norm{(x,y)}_\H\,,\\
    (iii)&\qquad B_\lambda=\partial\psi_\lambda\,, \quad \psi_\lambda:\H\to(-\infty,+\infty] \text{ proper, convex and l.s.c.}\,,
    \quad D(\psi_\lambda)\subseteq\V\\
    (iv)&\qquad \exists\,\ell_1, \ell_2>0:\quad \psi_\lambda(x,y)\geq\ell_1\norm{(x,y)}_{\V}^2-\ell_2\norm{(x,y)}_\H^2
    \quad\forall\,(x,y)\in D(\psi_\lambda)\,,\\
    (v)&\qquad A_\lambda=\partial\phi_\lambda\,, \quad \phi_\lambda:\H\to(-\infty,+\infty] \text{ proper, convex and l.s.c.}\,,
    \quad D(\phi_\lambda)=\H\,,\\
    (vi)&\qquad A_\lambda \quad\text{is bounded in } \H\,,\\
    (vii)&\qquad B_\lambda:\V\to\V^* \quad\text{is Lipschitz continuous and strongly monotone}\,.
  \end{align*}
\end{lem}
\begin{proof}
  It is clear that $A_\lambda$ and $B_\lambda$ are maximal monotone.
  By monotonicity of $\alpha_\lambda$ and $\alpha_{\Gamma\lambda}$ and the definition of $G_\lambda$,
  we have that 
  \[\begin{split}
  \left(A_\lambda(x,y), (x,y)\right)_\H&=\lambda\int_\Omega|x|^2 + \lambda\int_\Gamma|y|^2 + \int_\Omega\alpha_\lambda(x)x + 
  \int_\Gamma\alpha_{\Gamma\lambda}(y)y + \int_\Omega G_\lambda^{-1}(x)x\\
  &\geq\lambda\norm{(x,y)}_\H^2 + \lambda\int_\Omega|G_\lambda^{-1}(x)|^2 + \int_\Omega|\nabla G_\lambda^{-1}(x)|^2
  \geq\lambda\norm{(x,y)}_\H^2
  \end{split}\]
  for every $(x,y)\in\H$, from which the first condition. Secondly, for every $(x,y)\in\H$, the Lipschitz continuity of 
  $\alpha_\lambda$, $\alpha_{\Gamma\lambda}$ and the continuity of $G_\lambda^{-1}:\H\to W_{\bf n}$, we have
  \[
  \norm{A_\lambda(x,y)}_\H\leq\left(\lambda+\frac1\lambda\right)\norm{x,y}_\H + \norm{G_\lambda^{-1}(x)}_H
  \leq\left(\lambda+\frac1\lambda+\frac1{\sqrt\lambda}\right)\norm{x,y}_\H\,,
  \]
  from which the second condition. Furthermore, it is a standard matter to check that $(iii)$ holds with the choice
  $\psi_\lambda:\H\to[0,+\infty]$
  \[
  \psi_\lambda(x,y):=
  \begin{cases}
  \frac\lambda2\int_\Omega|x|^2 + \frac12\int_\Omega|\nabla x|^2 + 
  \frac\eps2\int_\Gamma|\nabla_\Gamma y|^2 + \int_\Omega\widehat{\beta}_\lambda(x)
  +\int_\Gamma\widehat{\beta}_{\Gamma\lambda}(y) &\text{ if } (x,y)\in \V\,,\\
  +\infty &\text{ otherwise}\,.
  \end{cases}
  \]
  It is clear that $D(\psi_\lambda)\subseteq\V$ and that, for every $(x,y)\in \V$,
  \[
  \psi_\lambda(x,y)\geq\frac\lambda2\int_\Omega|x|^2+\frac12\int_\Omega|\nabla x|^2+\frac\eps2\int_\Gamma|\nabla_\Gamma y|^2
  \geq\frac12\min\{1,\lambda,\eps\}\norm{(x,y)}_{\V}^2-\frac\eps2\norm{(x,y)}_\H^2
  \]
  and also condition $(iv)$ is proved. Moreover, it is readily seen that $(v)$ holds with
  \[
  \phi_\lambda(x,y):=\frac\lambda2\int_\Omega|x|^2 + \int_\Omega\widehat\alpha_\lambda(x)+
  F_\lambda^*(x) + \frac\lambda2\int_\Gamma|y|^2 + \int_\Gamma\widehat\alpha_{\Gamma\lambda}(y)\,,
  \qquad (x,y)\in\H\,,
  \]
  where $F_\lambda^*$ is the convex conjugate of the proper, convex, l.s.c.~function
  \[
    F_\lambda(x):=\begin{cases}
    \frac\lambda2\int_\Omega|x|^2 + \int_\Omega|\nabla x|^2 \quad&\text{if } x\in V\,,\\
    +\infty \quad&\text{if } x\in H\setminus V\,.
    \end{cases}
  \]
  Since $\partial\phi_\lambda=G_\lambda^{-1}$ is Lipschitz continuous on $H$, it is also clear that $D(\phi_\lambda)=H$,
  and $(v)$ is proved. Moreover, $(vi)$ is an easy consequence of the Lipschitz continuity of $\alpha_\lambda$,
  $\alpha_{\Gamma\lambda}$ and $G_\lambda^{-1}$ on $H$.
  Finally, let us focus on $(vii)$. In this case, we are looking at $B_\lambda$ as its weak 
  formulation $B_\lambda:\V\to\V^*$ given by
  \[
    \ip{B_\lambda(x,y)}{(z,w)}_{\V}=\lambda\int_\Omega xz + \int_\Omega\nabla x\cdot\nabla z
    +\int_\Omega \beta_\lambda(x)z + \eps\int_\Gamma\nabla_\Gamma y\cdot\nabla_\Gamma w
    +\int_\Gamma\beta_{\Gamma\lambda}(y)w\,.
  \]
  Hence, it follows by the Lipschitz continuity of $\beta_\lambda$ and $\beta_{\Gamma\lambda}$ that,
  for every $(x_1,y_1),(x_2,y_2)\in\V$
  \[\begin{split}
    &\norm{B_\lambda(x_1,y_1)-B_\lambda(x_2,y_2)}_{\V^*}\\
    &\qquad\leq\left(\lambda+\frac1\lambda\right)\norm{x_1-x_2}_H + \norm{\nabla(x_1-x_2)}_H 
    + \eps\norm{\nabla_\Gamma(y_1-y_2)}_{H_\Gamma} +\frac1{c\lambda}\norm{y_1-y_2}_{H_\Gamma}
  \end{split}\]
  from which the Lipschitz continuity of $B_\lambda$. Similarly, by the monotonicity of $\beta_\lambda$
  and $\beta_{\Gamma\lambda}$,
  \[\begin{split}
    &\ip{B_\lambda(x_1,y_1)-B_\lambda(x_2,y_2)}{(x_1,y_1)-(x_2,y_2)}_{\V}\\
    &\qquad\geq\lambda\norm{x_1-x_2}_H^2 + \norm{\nabla(x_1-x_2)}_H^2 + \eps\norm{\nabla_\Gamma(y_1-y_2)}_H^2
    \geq C_{\lambda\eps}\norm{(x_1,y_1)-(x_2,y_2)}_{\V}^2
  \end{split}\]
  for a certain positive constant $C_{\lambda\eps}$, from which the strong monotonicity of $B_\lambda$.
\end{proof}

Now, we fix $\lambda>0$ and we show that the approximated problem is well-posed. 
Given $(f,f_\Gamma)\in L^2(0,T; \H)$, Lemma~\ref{prop} and the hypotheses \eqref{g}--\eqref{u0} ensure that 
we can apply the existence result contained in \cite[Thm.~2.1]{colli-visin} and infer that there exists
\[
  (u_\lambda, v_\lambda) \in H^1(0,T; \H)\cap L^\infty(0,T; \V)\,, \qquad
  A_\lambda(\partial_tu_\lambda, \partial_tv_\lambda)\,,\; B_\lambda(u_\lambda,v_\lambda) \in L^2(0,T; \H)
\]
such that 
\begin{gather}
  \label{app1}
  A_\lambda\partial_t(u_\lambda,v_\lambda) + B_\lambda(u_\lambda,v_\lambda)=
  (g_\lambda,g_{\Gamma\lambda}) - (T_\lambda\pi(f), T_\lambda\pi_\Gamma(f_\Gamma))\quad\text{a.e.~in } (0,T)\,,\\
  \label{app2}
  (u_\lambda, v_\lambda)(0) = (u_0, u_{0|\Gamma})\,.
\end{gather}
Let us show that such solution $(u_\lambda, v_\lambda)$ is indeed unique and satisfies
useful estimates.
\begin{lem}
  \label{lem_app}
  For every $\lambda>0$, there exists $c_\lambda>0$ such that 
  \[
  \norm{\partial_tu_\lambda}_{L^2(0,T;H)} + \norm{\partial_tv_\lambda}_{L^2(0,T; H_\Gamma)}+
  \norm{u_\lambda}_{L^\infty(0,T; V)} + \norm{v_\lambda}_{L^\infty(0,T; V_{\Gamma})}\leq c_\lambda\,.
  \]
  Moreover, there is $c_\lambda'>0$ such that, for every $(f^i, f_\Gamma^i)\in L^2(0,T; \H)$,
  if $(u_\lambda^i, v_\lambda^i)$ are any respective solution to \eqref{app1}--\eqref{app2}, 
  $i=1,2$, we have
  \[\begin{split}
  \norm{\partial_t(u^1_\lambda-u_\lambda^2)}_{L^2(0,T; H)} &+ \norm{\partial_t(v^1_\lambda-v_\lambda^2)}_{L^2(0,T; H_\Gamma)}+
  \norm{u_\lambda^1-u_\lambda^2}_{L^\infty(0,T; V)} + \norm{v^1_\lambda-v_\lambda^2}_{L^\infty(0,T; V_{\Gamma})}\\
  &\leq c'_\lambda\left(\norm{f^1-f^2}_{L^2(0,T; H)} + \norm{f_\Gamma^1-f_\Gamma^2}_{L^2(0,T; H_\Gamma)}\right)\,.
  \end{split}\]
\end{lem}
\begin{proof}
  Testing \eqref{app1} by $\partial_t(u_\lambda, v_\lambda)$ and integrating on $(0,t)$, 
  thanks to the monotonicity of the operators $\alpha_\lambda$, $\alpha_{\Gamma\lambda}$ and $G_\lambda^{-1}$,
  using the Young inequality and the fact that $|T_\lambda|\leq\frac1\lambda$ we have
  \[\begin{split}
  &\lambda\int_0^t\norm{\partial_tu_\lambda(s)}_H^2\,ds 
  +\frac\lambda2\int_\Omega|u_\lambda(t)|^2
  +\frac12\int_\Omega|\nabla u_\lambda(t)|^2
  +\int_\Omega \widehat\beta_\lambda(u_\lambda(t))\\
  &\qquad+\lambda\int_0^t\norm{\partial_tv_\lambda(s)}_{H_\Gamma}^2\,ds
   + \frac\eps2\int_\Gamma|\nabla_\Gamma v_\lambda(t)|^2
  + \int_\Gamma\widehat\beta_{\Gamma\lambda}(v_\lambda(t))\\
  &\leq \frac\lambda2\norm{u_0}_H^2+ \frac12\norm{\nabla u_0}_H^2 + \frac\eps2\norm{\nabla_\Gamma u_{0|\Gamma}}_{H_\Gamma}^2
  +\int_\Omega\widehat\beta_\lambda(u_0) + \int_\Gamma\widehat\beta_{\Gamma\lambda}(u_{0|\Gamma})\\
  &\qquad+\int_0^t\!\!\int_\Omega \left(g_\lambda(s)-T_\lambda\pi(f(s))\right)\partial_tu_\lambda(s)\,ds + 
  \int_0^t\!\!\int_\Gamma \left(g_{\Gamma\lambda}(s)-T_\lambda\pi_\Gamma(f_\Gamma(s))\right)\partial_t v_\lambda(s)\,ds\\
  &\leq c_{\lambda}\norm{(u_0, u_{0|\Gamma})}_{\V}^2
  +\frac\lambda2\int_0^t\norm{\partial_tu_\lambda(s)}_H^2\,ds +
  \frac\lambda2\int_0^t\norm{\partial_tv_\lambda(s)}_{H_\Gamma}^2\,ds\\
  &\qquad+\frac1\lambda\norm{(g,g_{\Gamma})}^2_{L^2(0,T; \H)} + \frac1{\lambda^2}(|Q|+|\Sigma|)
\end{split}\]
for a certain $c_{\lambda}>0$,
so that rearranging the terms we obtain the first estimate.
Similarly, given $(f^i,f_\Gamma^i)$
and any respective solutions
$(u_\lambda^i, v_\lambda^i)$ to \eqref{app1}--\eqref{app2}, for $i=1,2$, taking the difference of \eqref{app1}
and testing by $\partial_t(u_\lambda^1-u_\lambda^2, v_\lambda^1-v_\lambda^2)$,
using the monotonicity of $\alpha_\lambda$, $\alpha_{\Gamma\lambda}$ and $G_\lambda^{-1}$, 
the Lipschitz continuity of $\beta_\lambda$, $\beta_{\Gamma\lambda}$, $\pi$, $\pi_\Gamma$ and $T_\lambda$,
an easy computation shows that 
\[\begin{split}
  \lambda&\int_0^t\norm{\partial_t(u_\lambda^1-u_\lambda^2)(s)}_H^2\,ds + 
  \lambda\int_0^t\norm{\partial_t(v_\lambda^1-v_\lambda^2)(s)}_{H_\Gamma}^2\,ds\\
  &\quad+\frac\lambda2\int_\Omega|(u_\lambda^1-u_\lambda^2)(t)|^2
  +\frac12\int_\Omega|\nabla(u_\lambda^1-u_\lambda^2)(t)|^2
  + \frac\eps2\int_\Gamma|\nabla(v_\lambda^1-v_\lambda^2)(t)|^2\\
  &\leq\int_0^t\!\!\int_\Omega\left(|\beta_\lambda(u_\lambda^1(s))-\beta_\lambda(u_\lambda^2(s))|+
  |T_\lambda\pi(f^1(s))-T_\lambda\pi(f^2(s))|\right)
  |\partial_t(u_\lambda^1-u_\lambda^2)(s)|\,ds\\
  &\quad+\int_0^t\!\!\int_\Gamma\left(|\beta_{\Gamma\lambda}(v_\lambda^1(s))-
  \beta_{\Gamma\lambda}(v_\lambda^2(s))|+
   |T_\lambda\pi_\Gamma(f_\Gamma^1(s))-T_\lambda\pi_\Gamma(f_\Gamma^2(s))|\right)
  |\partial_t(v_\lambda^1-v_\lambda^2)(s)|\,ds\\
  &\leq\frac1\lambda\int_0^t\!\!\int_\Omega|u^1_\lambda(s)-u_\lambda^2(s)||\partial_t(u_\lambda^1-u_\lambda^2)(s)|\,ds + 
  \frac1\lambda\int_0^t\!\!\int_\Gamma|v^1_\lambda(s)-v_\lambda^2(s)||\partial_t(v_\lambda^1-v_\lambda^2)(s)|\,ds\\
  &\quad+C_\pi\int_0^t\!\!\int_\Omega|f^1(s)-f^2(s)||\partial_t(u_\lambda^1-u_\lambda^2)(s)|\,ds
  +C_{\pi_\Gamma}\int_0^t\!\!\int_\Gamma|f^1_\Gamma(s)-f^2_\Gamma(s)||\partial_t(v_\lambda^1-v_\lambda^2)(s)|\,ds\\
  &\leq\frac\lambda2\int_0^t\norm{\partial_t(u_\lambda^1-u_\lambda^2)(s)}_H^2\,ds + 
  \frac\lambda2\int_0^t\norm{\partial_t(v_\lambda^1-v_\lambda^2)(s)}_{H_\Gamma}^2\,ds
  +\frac1{\lambda^2}\int_0^t\norm{u_\lambda^1(s)-u_\lambda^2(s)}_H^2\,ds\\
  &\quad+\frac1{\lambda^2}\int_0^t\norm{v_\lambda^1(s)-v_\lambda^2(s)}_{H_\Gamma}^2\,ds
  +\frac{C_\pi^2}\lambda\norm{f^1-f^2}^2_{L^2(0,T; H)}
  +\frac{C_{\pi_\Gamma}^2}\lambda\norm{f_\Gamma^1-f_\Gamma^2}^2_{L^2(0,T; H)}\,,
\end{split}\]
and the second inequality follows from the Gronwall lemma.
\end{proof}

Lemma~\ref{lem_app} ensures that, for any $\lambda>0$, it is well-defined the map
\[
  \Theta_\lambda: E_\lambda\to E_\lambda\,, \qquad (f,f_\Gamma)\mapsto (u_\lambda, v_\lambda)\,,
\]
where
\[\begin{split}
  E_\lambda&:=\left\{(x,y)\in H^1(0,T; \H)\cap L^\infty(0,T; \V): \right.\\
  &\qquad\left.\norm{\partial_tx}_{L^2(0,T;H)} + \norm{\partial_ty}_{L^2(0,T; H_\Gamma)}+
  \norm{x}_{L^\infty(0,T; V)} + \norm{y}_{L^\infty(0,T; V_{\Gamma})}\leq c_\lambda\right\}\,.
\end{split}\]
Since $E_\lambda$ is compact and convex in $L^2(0,T; \H)$ and $\Theta_\lambda$ is continuous 
on $L^2(0,T; \H)$ by Lemma~\ref{lem_app}, 
Shauder's fixed point theorem ensures that there is a fixed point $(u_\lambda, v_\lambda)\in E_\lambda$
for $\Theta_\lambda$. It is also clear by the second inequality in the previous lemma and the Gronwall lemma
that $(u_\lambda, v_\lambda)$ is also unique. As it is natural, we set $\mu_\lambda:=-G_\lambda^{-1}\partial_tu_\lambda$.

Let us collect the properties of $(u_\lambda, v_\lambda, \mu_\lambda)$ in the following lemmata.
The first result states precisely the regularities of the approximated solutions
under the weakest assumptions of Theorem~\ref{thm1} on the data, while the second
specifies some additional regularity provided by the strongest hypotheses 
of Theorems~\ref{thm2}--\ref{thm3}.
\begin{lem}\label{prop_app}
  Under the assumptions \eqref{g}--\eqref{u0} we have 
  \begin{gather*}
    u_\lambda \in H^1(0,T; H)\cap L^\infty(0,T; V)\cap L^2(0,T; W)\,,\\
    v_\lambda \in H^1(0,T; H_\Gamma)\cap L^\infty(0,T; V_{\Gamma})
    \cap L^2(0,T; W_\Gamma)\,,\\
    \mu_\lambda \in L^2(0,T; W_{\bf n})
  \end{gather*}
  and 
  \begin{align}
  \label{eq1_app}
  \partial_t u_\lambda + \lambda\mu_\lambda - \Delta\mu_\lambda = 0 \qquad&\text{in } Q\,,\\
  \label{eq2_app}
  \mu_\lambda = \lambda \partial_t u_\lambda + \alpha_\lambda(\partial_t u_\lambda) +\lambda u_\lambda - \Delta u_\lambda
  +\beta_\lambda(u_\lambda) + T_\lambda\pi(u_\lambda) - g_\lambda \qquad&\text{in } Q\,,\\
  u_\lambda=v_\lambda\,, \quad \partial_{\bf n}\mu_\lambda=0 \qquad&\text{in } \Sigma\,,\\
  \label{eq3_app}
  \lambda\partial_t v_\lambda + \alpha_{\Gamma\lambda}(\partial_tv_\lambda) + \partial_{\bf n}u_\lambda
  - \eps\Delta_\Gamma v_\lambda
  +\beta_{\Gamma\lambda}(v_\lambda) + T_\lambda\pi_\Gamma(v_\lambda) = g_{\Gamma\lambda} \qquad&\text{in } \Sigma\,,\\
  \label{init_app}
  u_\lambda(0)=u_0 \qquad&\text{in } \Omega\,,
\end{align}
\end{lem}
\begin{proof}
  Thanks to classical elliptic regularity results
  (see \cite[Thm.~3.2]{brezzi-gilardi}),
  the regularities of the approximated solutions
  $u_\lambda$ and $v_\lambda$ easily follow from the fact that $(u_\lambda,v_\lambda) \in E_\lambda$ and
  $B_\lambda(u_\lambda, v_\lambda)\in L^2(0,T; \H)$. Indeed, from this last condition it follows that 
  $\Delta u_\lambda \in L^2(0,T; H)$ and $\partial_{\bf n}u_\lambda -\eps\Delta v_\lambda \in L^2(0,T; H_\Gamma)$.
  The conditions 
  $u_\lambda \in L^\infty(0,T; V)$, $\Delta u_\lambda \in L^2(0,T; H)$ and $v_\lambda \in L^\infty(0,T; V_\Gamma)$ imply 
  that $u_\lambda \in L^2(0,T; H^{3/2}(\Omega))$, hence also $\partial_{\bf n}u_\lambda \in L^2(0,T; H_\Gamma)$.
  It follows then by comparison that $\Delta_\Gamma v_\lambda \in L^2(0,T; H_\Gamma)$, from which $v_\lambda\in 
  L^2(0,T; W_\Gamma)$
  and also $u_\lambda \in L^2(0,T; W)$. Finally, the regularity of $\mu$ is straightforward from the definition of
  $G_\lambda$, and \eqref{eq1_app}--\eqref{init_app} follow from the definition of $\Theta_\lambda$ itself.
\end{proof}

\begin{lem}
  \label{prop_app2}
  Under the further assumptions \eqref{g'}--\eqref{u0'_bis} we also have
  \begin{gather*}
    u_\lambda \in H^1(0,T; V)\cap L^2(0,T; H^3(\Omega))\cap C^0([0,T]; W)\cap C^1([0,T]; H)\,,\\
    v_\lambda \in H^1(0,T; V_{\Gamma\eps})\cap L^2(0,T; H^{3}(\Gamma))\cap C^0([0,T]; W_\Gamma)\cap C^1([0,T]; H_\Gamma)\,,\\
    \mu_\lambda \in L^2(0,T; H^3(\Omega))\cap C^0([0,T]; W_{\bf n})\,.
  \end{gather*}
\end{lem}
\begin{proof}
  Thanks to conditions $(v)$--$(vii)$ in Lemma~\ref{prop} and the hypotheses \eqref{g'}--\eqref{u0'_bis}, 
  the result \cite[Thm~2.2]{colli-visin} ensures that the range of the function $\Theta_\lambda$ is contained in
  $H^1(0,T; \V)$, hence $u_\lambda \in H^1(0,T;V)$ and $v_\lambda \in H^1(0,T; V_{\Gamma})$.
  Consequently, by comparison in \eqref{eq1_app}, we have $\mu_\lambda \in L^2(0,T; V)$, so that 
  $\mu_\lambda \in L^2(0,T; H^3(\Omega))$ by elliptic regularity. Moreover, by comparison in \eqref{eq2_app}--\eqref{eq3_app},
  thanks to \eqref{g'} and the fact that $\partial_t u_\lambda \in L^2(0,T; V)$ and $\partial_t v_\lambda \in L^2(0,T; V_{\Gamma})$,
  we deduce that $-\Delta u_\lambda\in L^2(0,T; V)$ and $\partial_{\bf n}u_\lambda - \eps\Delta v_\lambda \in 
  L^2(0,T; V_{\Gamma})$. 
  Since we have $\Delta u_\lambda \in L^2(0,T; V)$ and (by Lemma~\ref{prop_app}) $v \in L^2(0,T; W_\Gamma)$, then
  $u_\lambda \in L^2(0,T; H^{5/2}(\Omega))$ and $\partial_{\bf n}u_\lambda\in L^2(0,T; V_\Gamma)$.
  By difference then we deduce that $\Delta_\Gamma v_\lambda \in L^2(0,T; V_\Gamma)$, so that 
  $v_\lambda \in L^2(0,T; H^3(\Gamma))$ by elliptic regularity on the boundary, and consequently
  also $u_\lambda \in L^2(0,T; H^3(\Omega))$. Furthermore,
  we have $u_\lambda \in L^2(0,T; H^3(\Omega))\cap H^1(0,T; V)\embed C^0([0,T]; W)$
  and $v_\lambda \in L^2(0,T; H^3(\Gamma))\cap H^1(0,T; V_{\Gamma})\embed C^0([0,T]; W_\Gamma)$;
  in particular, we deduce that $\partial_{\bf n}u_\lambda \in C^0([0,T]; H^{1/2}(\Gamma))$. Hence, setting
  $z_\lambda:=g_\lambda-\lambda u_\lambda + \Delta u_\lambda - \beta_\lambda(u_\lambda) - T_\lambda\pi(u_\lambda)$
  and $w_\lambda:=g_{\Gamma\lambda}
  -\partial_{\bf n}u_\lambda - \beta_{\Gamma\lambda}(v_\lambda)-T_\lambda\pi_\Gamma(v_\lambda)$,
  from \eqref{eq1_app}--\eqref{eq3_app}
  we have that $A_\lambda(\partial_t u_\lambda, \partial_t v_\lambda)=(z_\lambda,w_\lambda)\in C^0([0,T];\H)$:
  since $A_\lambda^{-1}:\H\to\H$ is Lipschitz continuous, we infer that 
  $u_\lambda \in C^1([0,T]; H)$ and $v_\lambda \in C^1([0,T]; H_\Gamma)$, hence also $\mu_\lambda \in C^0([0,T]; W_{\bf n})$
  from \eqref{eq1_app}.
\end{proof}


\section{The first existence result}
\setcounter{equation}{0}
\label{proof1}

We present here the proof of the first main result.
Recall that here we are working under the assumptions \eqref{g}--\eqref{coerc1},
so that the regularity of the approximated solutions is the one specified in Lemma~\ref{prop_app}.
Since the passage to the limit will consist in letting $\lambda\searrow0$,
it is not restrictive to consider $\lambda\in(0,1)$ for example.

\subsection{The first estimate}\label{first}
Testing \eqref{eq1_app} by $\mu_\lambda$, \eqref{eq2_app} by $\partial_t u_\lambda$ and
taking the difference, by integration by parts we have that, for every $t\in(0,T)$, 
\[
  \begin{split}
    &\lambda\int_{Q_t}|\mu_\lambda|^2 + \int_{Q_t}|\nabla\mu_\lambda|^2
    +\lambda\int_{Q_t}|\partial_t u_\lambda|^2 + \int_{Q_t}\alpha_\lambda(\partial_t u_\lambda)\partial_t u_\lambda
    +\frac\lambda2\int_\Omega|u_\lambda(t)|^2 + \frac12\int_\Omega|\nabla u_\lambda(t)|^2\\
    &\qquad+\lambda\int_{\Sigma_t}|\partial_t v_\lambda|^2 
    + \int_{\Sigma_t}\alpha_{\Gamma\lambda}(\partial_t v_\lambda)\partial_t v_\lambda
    +\frac\eps2\int_\Sigma|\nabla_\Gamma v_\lambda(t)|^2 + \int_\Omega\widehat\beta_\lambda(u_\lambda(t))
    +\int_\Sigma\widehat\beta_{\Gamma\lambda}(v_\lambda(t))\\
    &=\frac\lambda2\int_\Omega|u_0|^2 + \frac12\int_\Omega|\nabla u_0|^2 + 
    \frac\eps2\int_\Sigma|\nabla_\Gamma u_{0|\Gamma}|^2
    +\int_\Omega\widehat\beta_\lambda(u_0) + \int_\Sigma\widehat\beta_{\Gamma\lambda}(u_{0|\Gamma})\\
    &\qquad+\int_{Q_t}\left(g_\lambda-T_\lambda\pi(u_\lambda)\right)\partial_t u_\lambda
    +\int_{\Sigma_t}\left(g_{\Gamma\lambda}-T_\lambda\pi_\Gamma(v_\lambda)\right)\partial_t v_\lambda\,.
  \end{split}
\]
Now, let $J_\lambda:=(I+\lambda\alpha)^{-1}:\erre\to\erre$
and $J_{\Gamma_\lambda}:=(I+\lambda\alpha_\Gamma)^{-1}:\erre\to\erre$ denote the resolvents of $\alpha$
and $\alpha_\Gamma$, respectively.
By elementary properties of maximal monotone graphs it 
is well known that $J_\lambda$ and $J_{\Gamma\lambda}$ are
contractions on $\erre$, and that $\alpha_\lambda(\cdot) \in \alpha(J_\lambda(\cdot))$ and 
$\alpha_{\Gamma\lambda}(\cdot)\in\alpha_\Gamma(J_{\Gamma\lambda}\cdot)$:
consequently, by the coercivity assumptions \eqref{coerc1} and \eqref{coerc2} we deduce that 
\begin{gather*}
  \alpha_{\lambda}(\partial_t u_\lambda)\partial_tu_\lambda =
  \alpha_{\lambda}(\partial_t u_\lambda)J_{\lambda}\partial_tu_\lambda
  +\lambda|\alpha_{\lambda}(\partial_t u_\lambda)|^2\geq
  a_1|J_{J_\lambda}\partial_t u_\lambda|^2 - a_2 + 
  \lambda|\alpha_{\lambda}(\partial_t u_\lambda)|^2\,,\\
  \alpha_{\Gamma\lambda}(\partial_t v_\lambda)\partial_tv_\lambda =
  \alpha_{\Gamma\lambda}(\partial_t v_\lambda)J_{\Gamma\lambda}\partial_tv_\lambda
  +\lambda|\alpha_{\Gamma\lambda}(\partial_t v_\lambda)|^2\geq
  b_1|J_{\Gamma\lambda}\partial_t v_\lambda|^2 - b_2 + 
  \lambda|\alpha_{\Gamma\lambda}(\partial_t v_\lambda)|^2\,.
\end{gather*}
Taking into account these relations, the left-hand side of the last inequality is bounded from below by
\[
  \begin{split}
    &\lambda\int_{Q_t}|\mu_\lambda|^2 + \int_{Q_t}|\nabla\mu_\lambda|^2
    +\lambda\int_{Q_t}|\partial_t u_\lambda|^2 + a_1\int_{Q_t}|J_{\lambda}\partial_t u_\lambda|^2
    +\lambda\int_{Q_t}|\alpha_{\lambda}(\partial_t u_\lambda)|^2\\
    &\qquad+\frac\lambda2\int_\Omega|u_\lambda(t)|^2 + \frac12\int_\Omega|\nabla u_\lambda(t)|^2
    +\lambda\int_{\Sigma_t}|\partial_t v_\lambda|^2 
    + b_1\int_{\Sigma_t}|J_{\Gamma\lambda}\partial_t v_\lambda|^2
    +\lambda\int_{\Sigma_t}|\alpha_{\Gamma\lambda}(\partial_t v_\lambda)|^2\\
    &\qquad+\frac\eps2\int_\Sigma|\nabla_\Gamma v_\lambda(t)|^2 + \int_\Omega\widehat\beta_\lambda(u_\lambda(t))
    +\int_\Sigma\widehat\beta_{\Gamma\lambda}(v_\lambda(t))
  \end{split}
\]
while the right-hand side
can be handled using the Young inequality by
\[
  \begin{split}
    &a_2|Q| + b_2|\Sigma| + 
    \frac12\norm{u_0}_V^2 + \frac\eps2\norm{u_{0|\Gamma}}_{V_{\Gamma}}^2 + \norm{\widehat\beta(u_0)}_{L^1(\Omega)}
    +\norm{\widehat\beta_{\Gamma}(u_{0|\Gamma})}_{L^1(\Gamma)}\\
    &\quad+\int_{Q_t}\left(g_\lambda-T_\lambda\pi(u_\lambda)\right)\partial_t u_\lambda
    +\int_{\Sigma_t}\left(g_{\Gamma\lambda}-T_\lambda\pi(v_\lambda)\right)\partial_t v_\lambda\\
    &\leq a_2|Q| + b_2|\Sigma| + 
    \frac12\norm{u_0}_V^2 + \frac\eps2\norm{u_{0|\Gamma}}_{V_{\Gamma}}^2 + \norm{\widehat\beta(u_0)}_{L^1(\Omega)}
    +\norm{\widehat\beta_{\Gamma}(u_{0|\Gamma})}_{L^1(\Gamma)}
    +\frac\delta2\int_{Q_t}|\partial_t u_\lambda|^2\\
    &\quad+\frac\delta2\int_{\Sigma_t}|\partial_t v_\lambda|^2
    +\frac1{\delta}\norm{g}^2_{L^2(0,T; H)} + \frac1{\delta}\norm{g_\Gamma}^2_{L^2(0,T; H_\Gamma)}
    +\frac{C_\pi^2}{\delta}\int_{Q_t}|u_\lambda|^2+\frac{C_{\pi_\Gamma}^2}{\delta}\int_{\Sigma_t}|v_\lambda|^2
  \end{split}
\]
for every $\delta>0$. Now, by definition of $\alpha_\lambda$ and $\alpha_{\Gamma\lambda}$,
\[
  \frac\delta2\int_{Q_t}|\partial_t u_\lambda|^2 \leq \delta\int_{Q_t}|J_\lambda\partial_t u_\lambda|^2
  +\delta\lambda^2\int_{Q_t}|\alpha_\lambda(\partial_t u_\lambda)|^2
\]
and
\[
  \frac\delta2\int_{\Sigma_t}|\partial_t v_\lambda|^2 \leq \delta\int_{\Sigma_t}|J_{\Gamma\lambda}\partial_t v_\lambda|^2
  +\delta\lambda^2\int_{\Sigma_t}|\alpha_{\Gamma\lambda}(\partial_t v_\lambda)|^2\,.
\]
Let us handle the last two terms on the right hand side.
Testing \eqref{eq1_app} by $\frac1{|\Omega|}$ we easily have
\[
  \partial_t(u_\lambda)_\Omega + \lambda(\mu_\lambda)_\Omega=0\,,
\]
which yields
\beq\label{mean}
  (u_\lambda(t))_\Omega=(u_0)_\Omega -\lambda\int_0^t(\mu_\lambda(s))_\Omega\,ds \quad\forall\,t\in[0,T]\,,\quad
  \forall\,\lambda>0\,.
\eeq
As a consequence, by the Poincar\'e inequality, an easy computation yields
\beq\label{normH}\begin{split}
  \norm{u_\lambda(t)}_H&\leq 
  \norm{u_\lambda(t)- (u_\lambda(t))_\Omega}_H+\norm{(u_0)_\Omega}_H
  +\lambda\int_0^t\norm{(\mu_\lambda(s))_\Omega}_H\,ds\\
  &\leq C\left(\norm{\nabla u_\lambda(t)}_H + \norm{u_0}_H + \lambda\int_0^t\norm{\mu_\lambda(s)}_H\,ds\right)
\end{split}\eeq
for a positive constant $C$ independent of $\lambda$, from which (updating $C$)
\[
  \int_{Q_t}|u_\lambda|^2\leq C\left(\int_{Q_t}|\nabla u_\lambda|^2 + \norm{u_0}_H^2 + \lambda^2\int_{Q_t}|\mu_\lambda|^2\right)\,.
\]
Moreover, by the Poincar\'e inequality on the boundary we also have 
\[
  \int_{\Sigma_t}|v_\lambda|^2 \leq C\int_{\Sigma_t}|\nabla_\Gamma v_\lambda|^2\,.
\]
Taking these considerations into account on the right hand side of the estimate
we obtain
\[
  \begin{split}
    &\lambda\int_{Q_t}|\mu_\lambda|^2 + \int_{Q_t}|\nabla\mu_\lambda|^2
    +\lambda\int_{Q_t}|\partial_t u_\lambda|^2 + a_1\int_{Q_t}|J_{\lambda}\partial_t u_\lambda|^2
    +\lambda\int_{Q_t}|\alpha_{\lambda}(\partial_t u_\lambda)|^2\\
    &\qquad+\frac\lambda2\int_\Omega|u_\lambda(t)|^2 + \frac12\int_\Omega|\nabla u_\lambda(t)|^2
    +\lambda\int_{\Sigma_t}|\partial_t v_\lambda|^2 
    + b_1\int_{\Sigma_t}|J_{\Gamma\lambda}\partial_t v_\lambda|^2
    +\lambda\int_{\Sigma_t}|\alpha_{\Gamma\lambda}(\partial_t v_\lambda)|^2\\
    &\qquad+\frac\eps2\int_\Sigma|\nabla_\Gamma v_\lambda(t)|^2 + \int_\Omega\widehat\beta_\lambda(u_\lambda(t))
    +\int_\Sigma\widehat\beta_{\Gamma\lambda}(v_\lambda(t))\\
   &\leq a_2|Q| + b_2|\Sigma| + 
    C\norm{u_0}_V^2 + \frac\eps2\norm{u_{0|\Gamma}}_{V_{\Gamma}}^2 + \norm{\widehat\beta(u_0)}_{L^1(\Omega)}
    +\norm{\widehat\beta_{\Gamma}(u_{0|\Gamma})}_{L^1(\Gamma)}
    +\delta\int_{Q_t}|J_\lambda\partial_t u_\lambda|^2\\
   &\qquad+\delta\lambda^2\int_{Q_t}|\alpha_\lambda(\partial_t u_\lambda)|^2
   +\delta\int_{\Sigma_t}|J_{\Gamma\lambda}\partial_t v_\lambda|^2
  +\delta\lambda^2\int_{\Sigma_t}|\alpha_{\Gamma\lambda}(\partial_t v_\lambda)|^2\\
  &\qquad+\frac1{\delta}\norm{g}^2_{L^2(0,T; H)} + \frac1{\delta}\norm{g_\Gamma}^2_{L^2(0,T; H_\Gamma)}
  + C_\delta\int_{Q_t}|\nabla u_\lambda|^2 + C_\delta\lambda^2\int_{Q_t}|\mu_\lambda|^2
 \end{split}
\]
where we have updated step by step the constant $C$ independent of $\lambda$
and $C_\delta>0$ depends only on
$\delta$.
Fix now $\delta:=\min\{\frac{a_1}2, \frac{b_1}2, \frac12\}$:
since it is not restrictive to consider $\lambda\in(0,\frac1{2C_\delta}]$, rearranging the terms and 
using the Gronwall lemma yields
\begin{gather}
  \label{est1}
  \norm{\nabla u_\lambda}_{L^\infty(0,T; H)} + \lambda^{1/2}\norm{u_\lambda}_{H^1(0,T; H)\cap L^\infty(0,T; H)} \leq C\,,\\
  \label{est2}
  \eps^{1/2}\norm{v_\lambda}_{L^\infty(0,T; V_{\Gamma})} + \lambda^{1/2}\norm{v_\lambda}_{H^1(0,T; H_\Gamma)} \leq C\,,\\
  \label{est3}
  \norm{J_{\lambda}\partial_tu_\lambda}_{L^2(0,T;H)} + \lambda^{1/2}\norm{\alpha_\lambda(\partial_t u_\lambda)}_{L^2(0,T; H)}
  \leq C\,,\\
  \label{est4}
  \norm{J_{\Gamma\lambda}\partial_tv_\lambda}_{L^2(0,T; H_\Gamma)} 
  + \lambda^{1/2}\norm{\alpha_{\Gamma\lambda}(\partial_t v_\lambda)}_{L^2(0,T; H_\Gamma)}\leq C\,,\\
  \label{est5}
  \lambda^{1/2}\norm{\mu_\lambda}_{L^2(0,T; H)} + \norm{\nabla\mu_\lambda}_{L^2(0,T; H)} \leq C\,,\\
  \label{est6}
  \norm{\widehat\beta_\lambda(u_\lambda)}_{L^\infty(0,T; L^1(\Omega))} + 
  \norm{\widehat\beta_{\Gamma\lambda}(v_\lambda)}_{L^\infty(0,T; L^1(\Gamma))} \leq C\,.
\end{gather}
From estimates \eqref{est1}, \eqref{est5}, condition \eqref{normH} and 
equation \eqref{eq1_app}, it follows that
\beq
  \label{est1'}
  \norm{u_\lambda}_{L^\infty(0,T; V)} + \norm{u_\lambda}_{H^1(0,T; V^*)}\leq C\,.
\eeq
Moreover, from \eqref{est3}, \eqref{est5} and the fact that 
$\partial_t u_\lambda=
\lambda\alpha_\lambda(\partial_t u_\lambda)+J_\lambda\partial_t u_\lambda$
(by definition of Yosida approximation), 
by comparison in \eqref{eq1_app} we have
\beq\label{est5'}
  \norm{\Delta \mu_\lambda}_{L^2(0,T; H)}\leq C\,.
\eeq
Finally, \eqref{alpha_sub} and \eqref{est3}--\eqref{est4} ensure that
\beq
  \label{est7}
  \norm{\alpha_\lambda(\partial_t u_\lambda)}_{L^2(0,T; H)} + 
  \norm{\alpha_{\Gamma\lambda}(\partial_tv_\lambda)}_{L^2(0,T; H_\Gamma)}\leq C\,.
\eeq

\subsection{The second estimate}\label{second}
We show here an additional estimate for $\mu_\lambda$ in the space $L^2(0,T; W_{\bf n})$. 
By \eqref{est5}, \eqref{est5'} and \eqref{V_eq}, it is enough to show that $(\mu_\lambda)_\Omega$ is bounded in $L^2(0,T)$ 
uniformly in $\lambda$.
To this end, we are inspired by the computations in \cite{col-gil-spr}.

We test \eqref{eq1_app} by $G_\lambda^{-1}(u_\lambda-(u_\lambda(t))_\Omega)$,
\eqref{eq2_app} by $u_\lambda-(u_\lambda(t))_\Omega$, take the difference, but not integrate in time: we deduce that,
for almost every $t\in(0,T)$,
\[\begin{split}
  &\int_\Omega|\nabla u_\lambda(t)|^2 + 
  \int_\Omega\beta_\lambda(u_\lambda(t))(u_\lambda(t)-(u_\lambda(t))_\Omega)
  +\frac\eps2\int_\Gamma|\nabla_\Gamma v_\lambda(t)|^2\\ &\qquad+ 
  \int_\Gamma\beta_{\Gamma\lambda}(v_\lambda(t))(v_\lambda(t)-(u_\lambda(t))_\Omega)
  =-\int_\Omega\partial_t u_\lambda(t)G_\lambda^{-1}(u_\lambda(t)-(u_\lambda(t))_\Omega)\\
  &\qquad+\int_\Omega\left(g_\lambda(t)-T_\lambda\pi(u_\lambda(t))-\lambda\partial_tu_\lambda(t)-\alpha_\lambda(u_\lambda(t))\right)
  (u_\lambda(t)-(u_\lambda(t))_\Omega)\\
  &\qquad+\int_\Gamma\left(g_{\Gamma\lambda}(t)-T_\lambda\pi_\Gamma(v_\lambda(t))-\lambda\partial_tv_\lambda(t)
  -\alpha_{\Gamma\lambda}(v_\lambda(t))\right)(v_\lambda(t)-(u_\lambda(t))_\Omega)\,.
\end{split}\]
Let us show that the right hand side is bounded in $L^2(0,T)$ uniformly in $\lambda$. It is clear that the last two terms are
bounded in $L^2(0,T)$ by the H\"older inequality and the estimates \eqref{est2}, \eqref{est1'} and \eqref{est7}. 
Moreover, by definition of $G_\lambda^{-1}$ it is immediate to check that
$(G_\lambda^{-1}(y))_\Omega=\frac1\lambda y_\Omega$ for every $y\in H$: hence,
we deduce that $(G_\lambda^{-1}(u_\lambda(t)- (u_\lambda(t))_\Omega))_\Omega=0$ and
by the Poincar\'e inequality we have
\[\begin{split}
  -\int_\Omega\partial_t u_\lambda(t)G_\lambda^{-1}(u_\lambda(t)-(u_0)_\Omega)
 & \leq\norm{\partial_t u_\lambda(t)}_{V^*}\norm{G_\lambda^{-1}(u_\lambda(t)-(u_\lambda(t))_\Omega)}_V\\
  &\leq C\norm{\partial_t u_\lambda(t)}_{V^*}\norm{\nabla G_\lambda^{-1}(u_\lambda(t)-(u_\lambda(t))_\Omega)}_H
\end{split}\]
for a positive constant $C$.
Now, for any $y\in H$ with $y_\Omega=0$, setting $y_\lambda:=G_\lambda^{-1}(y)\in W_{\bf n}$, we have
$\lambda y_\lambda - \Delta y_\lambda = y$, so that testing by $y_\lambda$ we infer that
\[
  \lambda\int_\Omega|y_\lambda|^2 + \int_\Omega|\nabla y_\lambda|^2=\int_\Omega yy_\lambda
  \leq \frac{1}{4\delta}\norm{y}_{V^*}^2 + \delta\norm{y_\lambda}_V^2\,,
\]
for every $\delta>0$, where
$\norm{y_\lambda}^2_V\leq C\norm{\nabla y_\lambda}^2_H$
for a positive constant $C$. 
Choosing $\delta=\frac1{2C}$ yields
\[
  \lambda\norm{G_\lambda^{-1}(y)}_H^2 + \norm{\nabla G_\lambda^{-1}(y)}_H^2\leq C\norm{y}_{V^*}^2
  \qquad\forall\,y\in H:\; y_\Omega=0\,,
\]
so that going back to the last inequality we have
\[
  -\int_\Omega\partial_t u_\lambda(t)G_\lambda^{-1}(u_\lambda(t)-(u_0)_\Omega)\leq
  C\norm{\partial_t u_\lambda(t)}_{V^*}\norm{u_\lambda(t)}_{V^*}\,.
\]
By \eqref{est1'} we deduce that also this last term is bounded in $L^2(0,T)$.

Now, by assumption \eqref{u0_mean} we know that $(u_0)_\Omega$ belongs to the interior of $D(\beta_\Gamma)$
(hence, also of $D(\beta)$ by \eqref{dom1}). This implies that there are two 
constants $k'_0, k_0''>0$ (depending only on $(u_0)_\Omega$) such that
\[
  \beta_\lambda(r)(r-(u_0)_\Omega)\geq k_0'|\beta_\lambda(r)| - k_0''\,, \quad
   \beta_{\Gamma\lambda}(r)(r-(u_0)_\Omega)\geq k_0'|\beta_\lambda(r)| - k_0' \qquad
   \forall\,r\in\erre
\]
(see for example \cite[p.~984]{col-gil-spr}, \cite[p.~908]{gil-mir-sch} and \cite[Prop.~A.1]{mir-zel}).
Moreover, note that by \eqref{mean} and \eqref{est5} we have
\[
  |(u_\lambda(t))_\Omega - (u_0)_\Omega| \leq \lambda \int_0^t|(\mu_\lambda(s))_\Omega|\,ds
  \leq C\lambda^{1/2} \qquad\forall\,t\in[0,T]\,.
\]
Consequently, we have 
\[\begin{split}
  \int_\Omega\beta_\lambda(u_\lambda(t))(u_\lambda(t)-(u_\lambda(t))_\Omega)&=
  \int_\Omega\beta_\lambda(u_\lambda(t))(u_\lambda(t)-(u_0)_\Omega)
  +\int_\Omega\beta_\lambda(u_\lambda(t))((u_0)_\Omega - (u_\lambda(t))_\Omega)\\&\geq
  k_0'\int_\Omega|\beta_\lambda(u_\lambda(t))| - k_0''|\Omega| - C\lambda^{1/2}\int_\Omega|\beta_\lambda(u_\lambda(t))|
\end{split}\]
and similarly
\[
  \int_\Gamma\beta_{\Gamma\lambda}(v_\lambda(t))(v_\lambda(t)-(u_\lambda(t))_\Omega)
  \geq
  k_0'\int_\Gamma|\beta_{\Gamma\lambda}(v_\lambda(t))| - k_0''|\Gamma| - C\lambda^{1/2}\int_\Gamma|\beta_{\Gamma\lambda}(v_\lambda(t))|
\]
Putting this information together, we deduce that 
\[
  \norm{\beta_\lambda(u_\lambda)}_{L^2(0,T; L^1(\Omega))} + 
  \norm{\beta_{\Gamma\lambda}(v_\lambda)}_{L^2(0,T; L^1(\Gamma))} \leq C\,. 
\]
Hence, testing \eqref{eq2_app} by $\pm1$ we have
\[\begin{split}
\pm |\Omega|(\mu_\lambda)_\Omega&\leq \int_\Omega|\beta_\lambda(u_\lambda)| + \int_\Gamma|\beta_{\Gamma\lambda}(v_\lambda)|
+\int_\Omega\left|\lambda\partial_t u_\lambda
+\alpha_\lambda(\partial_t u_\lambda) + \lambda u_\lambda + T_\lambda\pi(u_\lambda)-g_\lambda\right|\\
&+\int_\Gamma\left|\lambda\partial_t v_\lambda
+\alpha_{\Gamma\lambda}(\partial_t v_\lambda) + \lambda v_\lambda + T_\lambda\pi_\Gamma(v_\lambda)-
g_{\Gamma\lambda}\right|\,,
\end{split}
\]
where the right hand side is bounded in $L^2(0,T)$ by the estimates already computed and by \eqref{est1}--\eqref{est2}
and \eqref{est1'}--\eqref{est7}. Hence, we have that
\beq\label{est9}
  \norm{\mu_\lambda}_{L^2(0,T; {W_{\bf n}})}\leq C\,.
\eeq

\subsection{The third estimate}\label{third}
We test \eqref{eq2_app} by $\beta_\lambda(u_\lambda)$:
integrating by parts yields
\[\begin{split}
  &\lambda\int_\Omega\widehat\beta_\lambda(u_\lambda(t)) +
   \lambda\int_{Q_t}\beta_\lambda(u_\lambda)u_\lambda
  +\int_{Q_t}\beta_\lambda'(u_\lambda)|\nabla u_\lambda|^2 + \int_{Q_t}|\beta_\lambda(u_\lambda)|^2\\
  &\qquad+\lambda\int_\Gamma\widehat\beta_\lambda(v_\lambda(t)) + 
  \int_{\Sigma_t}\beta_\lambda'(v_\lambda)|\nabla_\Gamma v_\lambda|^2
  +\int_{\Sigma_t}\beta_{\Gamma\lambda}(v_\lambda)\beta_\lambda(v_\lambda)
  =\lambda\int_\Omega\widehat\beta_\lambda(u_0) + \lambda\int_\Gamma\widehat\beta_\lambda(u_{0|\Gamma})\\
  &\qquad+\int_{Q_t}\left(g_\lambda-T_\lambda\pi(u_\lambda)-\alpha_\lambda(\partial_t u_\lambda)\right)\beta_\lambda(u_\lambda)
  +\int_{\Sigma_t}\left(g_{\Gamma\lambda}-T_\lambda\pi_\Gamma(v_\lambda)-\alpha_{\Gamma\lambda}(\partial_t v_\lambda)\right)
  \beta_\lambda(v_\lambda)
\end{split}\]
By the Young inequality, the estimates \eqref{est1}--\eqref{est2} and \eqref{est1'}--\eqref{est7}, the hypotheses
\eqref{g}--\eqref{u0} and the monotonicity of $\beta$ and $\beta_\Gamma$, we infer that for every $\delta>0$, we have
\[
  \frac12\int_{Q_t}|\beta_\lambda(u_\lambda)|^2 + \int_{\Sigma_t}\beta_{\Gamma\lambda}(v_\lambda)\beta_\lambda(v_\lambda)
  \leq C_\delta + \norm{\widehat\beta(u_0)}_{L^1(\Omega)} + \norm{\widehat\beta_\Gamma(u_{0|\Gamma})}_{L^1(\Gamma)}
  + \delta\int_{\Sigma_t}|\beta_\lambda(v_\lambda)|^2
\]
for a positive constant $C_\delta$, independent of $\lambda$. Now, by the assumption \eqref{dom1} and 
\cite[Lemma~4.4]{cal-colli}, recalling the definition of $\beta_\lambda$ and $\beta_{\Gamma\lambda}$,
it follows that
\[
  |\beta_\lambda(r)|\leq c\left(|\beta_{\Gamma\lambda}(r)|+1\right) \quad\forall\,r\in\erre\,.
\]
Hence, substituting in the last inequality and using the Young inequality we get (updating the constant $C_\delta$
at each step)
\[
  \frac12\int_{Q}|\beta_\lambda(u_\lambda)|^2 + \frac1c\int_{\Sigma}|\beta_\lambda(v_\lambda)|^2\leq
  C_\delta+\delta\int_\Sigma|\beta_\lambda(v_\lambda)|^2 + \int_\Sigma|\beta_\lambda(v_\lambda)|\leq
  C_\delta + 2\delta\int_\Sigma|\beta_\lambda(v_\lambda)|^2\,.
\]
Choosing $\delta:=\frac1{4c}$, we infer that
\beq\label{est10}
  \norm{\beta_\lambda(u_\lambda)}_{L^2(0,T; H)} + \norm{\beta_\lambda(v_\lambda)}_{L^2(0,T; H_\Gamma)} \leq C\,.
\eeq
By comparison in \eqref{eq2_app}, recalling also \eqref{est1}, \eqref{est7} and \eqref{est9}, we deduce that 
\beq
  \label{est11}
  \norm{\Delta u_\lambda}_{L^2(0,T; H)}\leq C\,.
\eeq
Hence, thanks to the classical results on elliptic regularity (see \cite[Thm.~3.2]{brezzi-gilardi}), 
\eqref{est2}, \eqref{est1'} and \eqref{est11} yield
\beq\label{est12}
  \eps^{1/2}\norm{u_\lambda}_{L^2(0,T; H^{3/2}(\Omega))} + \eps^{1/2}\norm{\partial_{\bf n}u_\lambda}_{L^2(0,T; H_\Gamma)} \leq C\,,
\eeq
and by comparison in \eqref{eq3_app} also
\[
  \norm{-\eps^{3/2}\Delta_\Gamma v_\lambda + \eps^{1/2}\beta_{\Gamma\lambda}(v_\lambda)}_{L^2(0,T; H_\Gamma)}\leq C\,.
\]
Now, since the operators $-\Delta_\Gamma$ and $\beta_{\Gamma\lambda}$ are monotone on $H_\Gamma$, 
testing $-\eps^{3/2}\Delta_\Gamma v_\lambda + \eps^{1/2}\beta_{\Gamma\lambda}(v_\lambda)$ 
by either $-\eps^{3/2}\Delta_\Gamma v_\lambda$ or $\eps^{1/2}\beta_{\Gamma\lambda}(v_\lambda)$,
integrating by parts on $\Gamma$, using monotonicity, the last estimate and the Young inequality
implies by a classical argument that
\beq\label{est13}
  \eps^{3/2}\norm{\Delta_\Gamma v_\lambda}_{L^2(0,T; H_\Gamma)} + 
  \eps^{1/2}\norm{\beta_{\Gamma\lambda}(v_\lambda)} \leq C\,.
\eeq

\subsection{The passage to the limit}
\label{limit}
In this section, we pass to the limit in the approximated problem \eqref{eq1_app}--\eqref{init_app} and we prove
the existence of a solution for the original problem.

First of all, thanks to the estimates \eqref{est1}--\eqref{est13}, there are
\begin{gather*}
  u \in L^\infty(0,T; V)\cap L^2(0,T; W)\,, \qquad
  v \in L^\infty(0,T; V_{\Gamma})\cap L^2(0,T; W_\Gamma)\,,\qquad
  \mu \in L^2(0,T; W_{\bf n})\,,\\
  \eta, \xi \in L^2(0,T; H)\,, \qquad \eta_\Gamma, \xi_\Gamma \in L^2(0,T; H_\Gamma)\,,
\end{gather*}
such that, along a subsequence that we still denote by $\lambda$ for simplicity,
\begin{gather}
  \label{conv1}
  u_\lambda \wstarto u \quad\text{in } L^\infty(0,T; V)\,, \qquad
  u_\lambda \wto u \quad\text{in } L^2(0,T; W)\,,\\
  \label{conv2}
  v_\lambda \wstarto u \quad\text{in } L^\infty(0,T; V_{\Gamma})\,,
  \qquad v_\lambda \wto v \quad\text{in } L^2(0,T; W_\Gamma)\,,\\
  \label{conv4}
  \mu_\lambda \to \mu \quad\text{in } L^2(0,T; W_{\bf n})\,,\\
  \label{conv5}
  \alpha_\lambda(\partial_t u_\lambda) \wto \eta \quad\text{in } L^2(0,T; H)\,, \qquad
  \alpha_{\Gamma\lambda}(\partial_t v_\lambda) \wto \eta_\Gamma \quad\text{in } L^2(0,T; H_\Gamma)\,,\\
  \label{conv6}
  \beta_\lambda(u_\lambda) \wto \xi \quad\text{in } L^2(0,T; H)\,, \qquad
  \beta_{\Gamma\lambda}(v_\lambda) \wto \xi_\Gamma \quad\text{in } L^2(0,T; H_\Gamma)
\end{gather}
and
\beq\label{conv7}
  \lambda u_\lambda \to 0 \quad\text{in } H^1(0,T; H)\,, \quad
  \lambda v_\lambda \to 0 \quad\text{in } H^1(0,T; H_\Gamma)\,, \quad
  \lambda\mu_\lambda\to0 \quad\text{in } L^2(0,T; H)\,.
\eeq
Moreover, noting that, by definition of Yosida approximation,
\[
  |\partial_t u_\lambda - J_\lambda\partial_t u_\lambda|=\lambda|\alpha_\lambda(\partial_t u_\lambda)|\,, \qquad
  |\partial_t v_\lambda - J_{\Gamma\lambda}\partial_t v_\lambda|=\lambda|\alpha_{\Gamma\lambda}(\partial_t v_\lambda)|\,,
\]
it is readily seen that \eqref{est3}--\eqref{est4} imply that $u \in H^1(0,T; H)$, $v \in H^1(0,T; H_\Gamma)$ and 
\beq\label{conv8}
  J_\lambda\partial_t u_\lambda \wto \partial_t u \quad\text{in } L^2(0,T; H)\,, \qquad
  J_{\Gamma\lambda}\partial_t v_\lambda \wto \partial_t v \quad\text{in } L^2(0,T; H_\Gamma)\,.
\eeq
It is clear that $u_{|\Gamma}=v$.
Moreover, since the inclusion $\V\embed \H$ is compact, by the classical compactness results for functions
with values in Banach spaces (see \cite[Cor.~4, p.~85]{simon}), we have
\beq\label{conv9}
  u_\lambda \to u \quad\text{in } C^0([0,T]; H)\,, \qquad v_\lambda\to v \quad\text{in } C^0([0,T]; H_\Gamma)\,,
\eeq
which together with \eqref{conv6} and the
strong-weak closure of the maximal monotone operators
$\beta$ and $\beta_\Gamma$ 
ensure that 
\[
  \xi \in \beta(u) \quad\text{a.e.~in } Q\,, \qquad \xi_\Gamma\in \beta_\Gamma(v) \quad\text{a.e.~in } \Sigma\,.
\]
Furthermore, by the Lipschitz continuity of $T_\lambda$, $\pi$ and $\pi_\Gamma$,
using the strong convergences of $u_\lambda$ and $v_\lambda$ it is a standard matter to check that
\[
  T_\lambda\pi(u_\lambda)\to \pi(u) \quad\text{in } L^2(0,T; H)\,, \qquad
  T_\lambda\pi_\Gamma(v_\lambda)\to \pi_\Gamma(v) \quad\text{in } L^2(0,T; H_\Gamma)\,.
\]

Taking this information into account and letting $\lambda\searrow0$ in \eqref{eq1_app}--\eqref{init_app}, we get
\begin{gather}
  \label{lim1}
  \partial_t u - \Delta \mu = 0\,, \\
  \label{lim2}
  \mu=\eta - \Delta u + \xi + \pi(u) - g\,,\qquad
  \eta_\Gamma + \partial_{\bf n} u - \eps\Delta_\Gamma v + \xi_\Gamma + \pi_\Gamma(v) = g_\Gamma\,.
\end{gather}

The last thing that we have to prove is that $\eta\in\alpha(\partial_t u)$ a.e.~in $Q$ and
$\eta_\Gamma\in\alpha_\Gamma(\partial_t v)$ a.e.~in~$\Sigma$.
To this end, performing the same test as in Section~\ref{first}, one can easily infer that
\[\begin{split}
  &\int_Q|\nabla \mu_\lambda|^2 
  + \int_Q\alpha_\lambda(\partial_t u_\lambda)\partial_t u_\lambda
  +\frac12\int_\Omega|\nabla u_\lambda(T)|^2 
  +\int_\Omega\widehat\beta_\lambda(u_\lambda(t))\\
  &\qquad+ \int_\Sigma\alpha_{\Gamma\lambda}(\partial_t v_\lambda)\partial_t v_\lambda
  +\frac\eps2\int_\Gamma|\nabla_\Gamma v_\lambda(t)|^2
  +\int_\Gamma\widehat\beta_{\Gamma\lambda}(v_\lambda(t))\\
  &\leq\frac\lambda2\int_\Omega|u_0|^2 + \frac12\int_\Omega|\nabla u_0|^2 + \int_\Omega\widehat\beta(u_0)
  +\frac\eps2\int_\Gamma|\nabla_\Gamma u_{0|\Gamma}|^2 + \int_\Gamma\widehat\beta_\Gamma(u_{0|\Gamma})\\
  &\qquad+\int_Q\left(g_\lambda-T_\lambda\pi(u_\lambda)\right)\partial_t u_\lambda
  +\int_\Sigma\left(g_{\Gamma\lambda}-T_\lambda\pi(v_\lambda)\right)\partial_t v_\lambda\,.
\end{split}\]
Now, since $u\in H^1(0,T; H)$, $v\in H^1(0,T; H_\Gamma)$,
$\xi \in \beta(u)$ a.e.~in $Q$ and $\xi_\Gamma \in \beta_\Gamma(v)$ a.e.~in $\Sigma$, 
by \cite[Lemma~3.3]{brezis} the functions
\[
  t \mapsto \int_\Omega \widehat\beta(u(t))\,, \qquad t\mapsto \int_\Gamma \widehat\beta_\Gamma(v(t))\,,
\]
are absolutely continuous on $[0,T]$ with derivatives given by $(\xi,\partial_t u)_H$ and $(\xi_\Gamma, \partial_t v)_{H_\Gamma}$,
respectively. Moreover, the strong convergence of $u_\lambda$ and $v_\lambda$ together with
\cite[Prop.~2.11]{brezis} ensure that
\[
  \int_\Omega\widehat\beta_\lambda(u_\lambda(T))\to \int_\Omega\widehat\beta(u(T))\,, \qquad
  \int_\Omega\widehat\beta_{\Gamma\lambda}(v_\lambda(T))\to \int_\Omega\widehat\beta_\Gamma(v(T))\,.
\]
Hence, by \eqref{conv1}--\eqref{conv7} and the weak lower semicontinuity of the convex integrands,
we infer
\[\begin{split}
  &\limsup_{\lambda\searrow0}\left[\int_Q\alpha_\lambda(\partial_t u_\lambda)\partial_t u_\lambda
  +\int_\Sigma\alpha_{\Gamma\lambda}(\partial_t v_\lambda)\partial_t v_\lambda\right]\\
  &\leq\frac12\int_\Omega|\nabla u_0|^2 
  +\frac\eps2\int_\Gamma|\nabla_\Gamma u_{0|\Gamma}|^2
  + \int_\Omega\widehat\beta(u_0)
  + \int_\Gamma\widehat\beta_\Gamma(u_{0|\Gamma})
  +\int_Q\left(g-\pi(u)\right)\partial_t u\\
  &\quad+\int_\Sigma\left(g_\Gamma-\pi(v)\right)\partial_t v
  -\liminf_{\lambda\searrow0}\left[\int_Q|\nabla\mu_\lambda|^2 + \frac12\int_\Omega|\nabla u_\lambda(T)|^2
  +\int_\Omega\widehat\beta(u_\lambda(T)) + \int_\Gamma \widehat\beta_\Gamma(v_\lambda(T))\right]\\
  &\leq\frac12\int_\Omega|\nabla u_0|^2 
  +\frac\eps2\int_\Gamma|\nabla_\Gamma u_{0|\Gamma}|^2
  + \int_\Omega\widehat\beta(u_0)
  + \int_\Gamma\widehat\beta_\Gamma(u_{0|\Gamma})
  +\int_Q\left(g-\pi(u)\right)\partial_t u\\
  &\quad+\int_\Sigma\left(g_\Gamma-\pi(v)\right)\partial_t v
  -\int_Q|\nabla\mu|^2 - \frac12\int_\Omega|\nabla u(T)|^2
  -\int_\Omega\widehat\beta(u(T)) - \int_\Gamma \widehat\beta_\Gamma(v(T))
\end{split}\]
Now, testing equation \eqref{lim1} by $\mu$, 
the first equation in \eqref{lim2} by $\partial_t u$ and taking the difference,
it is a standard matter to check that 
the right hand side of the last inequality coincides with
\[
  \int_Q\eta\partial_t u + \int_\Sigma \eta_\Gamma \partial_t v\,,
\]
so that
\[
  \limsup_{\lambda\searrow0}\left[\int_Q\alpha_\lambda(\partial_t u_\lambda)\partial_t u_\lambda
  +\int_\Sigma\alpha_{\Gamma\lambda}(\partial_t v_\lambda)\partial_t v_\lambda\right]\leq
  \int_Q\eta\partial_t u + \int_\Sigma \eta_\Gamma \partial_t v\,.
\]
This implies by a classical argument on maximal monotone operators
that $\eta \in \alpha(\partial_tu)$ a.e.~in $Q$ and $\eta_\Gamma \in \alpha_\Gamma(\partial_t v)$
a.e.~in $\Sigma$. This concludes the proof of Theorem~\ref{thm1}.


\section{The second existence result}
\setcounter{equation}{0}
\label{proof2}

We present here the proof of the second main result of the paper.
Recall that we are working now under the stronger conditions \eqref{g'}--\eqref{u0'_bis},
so that the regularity of the approximated solutions
is the one given by Lemma~\ref{prop_app2}.

\subsection{The first estimate}\label{first'}
We proceed as in Section~\ref{first}, using the monotonicity of $\alpha_\lambda$ 
on the left hand side. For every $t\in[0,T]$ we obtain
\[
  \begin{split}
    &\lambda\int_{Q_t}|\mu_\lambda|^2 + \int_{Q_t}|\nabla\mu_\lambda|^2
    +\lambda\int_{Q_t}|\partial_t u_\lambda|^2 
    +\frac\lambda2\int_\Omega|u_\lambda(t)|^2 + \frac12\int_\Omega|\nabla u_\lambda(t)|^2
    + \int_\Omega\widehat\beta_\lambda(u_\lambda(t))\\
    &\qquad+\lambda\int_{\Sigma_t}|\partial_t v_\lambda|^2
    +\int_{\Sigma_t}\alpha_{\Gamma\lambda}(\partial_t v_\lambda)\partial_t v_\lambda
    +\frac\eps2\int_\Sigma|\nabla_\Gamma v_\lambda(t)|^2 
    +\int_\Sigma\widehat\beta_{\Gamma\lambda}(v_\lambda(t))\\
    &\leq\frac\lambda2\int_\Omega|u_0|^2 + \frac12\int_\Omega|\nabla u_0|^2 + 
    \frac\eps2\int_\Sigma|\nabla_\Gamma u_{0|\Gamma}|^2
    +\int_\Omega\widehat\beta_\lambda(u_0) + \int_\Sigma\widehat\beta_{\Gamma\lambda}(u_{0|\Gamma})\\
    &\qquad+\int_{Q_t}\left(g_\lambda-T_\lambda\pi(u_\lambda)\right)\partial_t u_\lambda
    +\int_{\Sigma_t}\left(g_{\Gamma\lambda}-T_\lambda\pi_\Gamma(v_\lambda)\right)\partial_t v_\lambda\,.
  \end{split}
\]
Now, in order to handle the terms on the boundary,
we proceed exactly as in Section~\ref{first} using the coercivity of $\alpha_\Gamma$
on the left hand side
combined with the weighted Young inequality on the last term in right-hand side.
Furthermore, thanks to hypothesis \eqref{g'} and \eqref{eq1_app},
integrating by parts and taking into account that $\lambda\in(0,1)$
and the Lipschitz continuity of $T_\lambda$ and $\pi$, we have
\[\begin{split}
  &\int_{Q_t}\left(g_\lambda-T_\lambda\pi(u_\lambda)\right)\partial_t u_\lambda=
  \int_{Q_t}g_\lambda\partial_t u_\lambda + \int_{Q_t}T_\lambda\pi(u_\lambda)\left(\lambda\mu_\lambda-\Delta\mu_\lambda\right)\\
  &=-\int_0^t\ip{\partial_t g(s)}{u_\lambda(s)}_V\,ds + \int_\Omega g_\lambda(t)u_\lambda(t) - \int_\Omega g_\lambda(0)u_0\\
  &\qquad+\lambda\int_{Q_t}T_\lambda\pi(u_\lambda)\mu_\lambda
  +\int_{Q_t}\nabla T_\lambda\pi(u_\lambda)\cdot \nabla \mu_\lambda\\
   &\leq \frac12\norm{g}^2_{H^1(0,T; V^*)} + \frac12\norm{u_\lambda}_{L^2(0,t; V)}
   +\frac1{4\delta}\norm{g}_{L^\infty(0,T; V^*)}^2 + \delta\norm{u_\lambda(t)}_V^2
   +\norm{g}_{L^\infty(0,T; V^*)}\norm{u_0}_V\\
   &\qquad+\frac\lambda2\int_{Q_t}|\mu_\lambda|^2+ \frac12\int_{Q_t}|\nabla\mu_\lambda|^2 
  + \frac{C_\pi^2+1}2\norm{u_\lambda}_{L^2(0,t; V)}
\end{split}\]
for every $\delta>0$. Now, we write
\[
  \norm{u_\lambda}_V^2=\norm{\nabla u_\lambda}^2_H + \norm{u_\lambda}^2_H\,,
\]
where the first term can be handled using Gronwall's lemma and the second by \eqref{normH}.
Hence, choosing $\delta$ small enough
and rearranging the terms, thanks to the Gronwall lemma
we still obtain the estimates \eqref{est1}--\eqref{est2} and \eqref{est4}--\eqref{est1'}.

\subsection{The second estimate}\label{second'}
First of all, in order to perform this estimate, 
we need to identify the initial values at $t=0$ of $\partial_t u_\lambda$, $\partial_t v_\lambda$ and $\mu_\lambda$:
to this end, it is natural to require that these satisfy the system \eqref{eq1_app}--\eqref{eq3_app} at $t=0$.
We have the following result.
\begin{lem}\label{init_reg}
  There is a unique triplet $(u_{0\lambda}', v_{0\lambda}', \mu_{0\lambda})\in H\times H_{\Gamma}\times W_{\bf n}$ 
  such that
  \[
    \begin{cases}
    u_{0\lambda}' + \lambda\mu_{0\lambda} - \Delta\mu_{0\lambda}=0 \quad&\text{in } \Omega\,,\\
    \mu_{0\lambda} = \lambda u_{0\lambda}' + \alpha_\lambda(u_{0\lambda}') + \lambda u_0
    -\Delta u_0 + \beta_\lambda(u_0) + T_\lambda\pi(u_0) - g_\lambda(0)\quad&\text{in } \Omega\,,\\
    \lambda v_{0\lambda}' + \alpha_{\Gamma\lambda}(v_{0\lambda}') + \partial_{\bf n}u_0
    -\eps\Delta_\Gamma u_{0|\Gamma} + \beta_{\Gamma\lambda}(u_{0|\Gamma})
    +T_\lambda\pi_\Gamma(u_{0|\Gamma})=g_{\Gamma\lambda}(0) \quad&\text{in } \Gamma\,.
    \end{cases}
  \]
  Furthermore, there exists $C>0$, independent of $\lambda$, such that 
  \[
    \lambda\norm{\mu_{0\lambda}}_H^2 + \norm{\nabla\mu_{0\lambda}}_H^2
    +\lambda\norm{u_{0\lambda}'}_H^2 
    + \norm{\widehat{\alpha_\lambda^{-1}}(\alpha_\lambda(u_{0\lambda}'))}_{L^1(\Omega)}
    +\lambda\norm{v_{0\lambda}'}_H^2 
    +\norm{\widehat{\alpha_{\Gamma\lambda}^{-1}}(\alpha_{\Gamma\lambda}(v_{0\lambda}'))}_{L^1(\Gamma)} \leq C\,.
  \]
\end{lem}
\begin{proof}
  Setting $z_{0\lambda}:=g_\lambda(0)-T_\lambda\pi(u_0)-\beta_\lambda(u_0)+\Delta u_0 - \lambda u_0$
  and $w_{0\lambda}:=g_{\Gamma\lambda}(0)-T_\lambda\pi_\Gamma(u_{0|\Gamma})-\beta_{\Gamma\lambda}(u_{0|\Gamma})
  +\eps\Delta_\Gamma u_{0|\Gamma}-\partial_{\bf n}u_0$, by the hypothesis \eqref{u0'} we have $(z_{0\lambda}, w_{0\lambda})\in \H$.
  Moreover, the system which we are interested in reduces to $A_\lambda(u_{0\lambda}', v_{0\lambda}')=(z_{0\lambda}, w_{0\lambda})$,
  with $\mu_{0\lambda}=-G_\lambda^{-1}(u_{0\lambda}')$. Since $A_\lambda$ is bi-Lipschitz continuous on $\H$, 
  there is a unique pair $(u_{0\lambda}', v_{0\lambda}')\in\H$ solving the system with 
  $\mu_{0\lambda}=-G_\lambda^{-1}(u_{0\lambda}')\in W_{\bf n}$ by definition of $G_\lambda$.
  Furthermore, testing the first equation by $\mu_{0\lambda}$, the second by $u_{0\lambda}'$, taking the difference
  and recalling the hypotheses \eqref{u0'}--\eqref{u0'_bis}, we have 
  \[\begin{split}
  \lambda\int_\Omega|\mu_{0\lambda}|^2&+\int_\Omega|\nabla\mu_{0\lambda}|^2
  +\lambda\int_\Omega|u_{0\lambda}'|^2 + \int_\Omega\alpha_\lambda(u_{0\lambda}')u_{0\lambda}'
  +\lambda\int_\Gamma|v_{0\lambda}'| + \int_\Gamma\alpha_{\Gamma\lambda}(v_{0\lambda}')v_{0\lambda}'\\
  &\leq\int_\Omega z_{0\lambda}u_{0\lambda}' + \int_\Gamma w_{0\lambda}v_{0\lambda}'\,.
  \end{split}\]
  On the left hand side, we use \eqref{coerc2} and the fact that $\alpha_{\Gamma\lambda}\in\alpha_\Gamma(J_{\Gamma\lambda})$
  to infer that 
  \[
    \int_\Gamma\alpha_{\Gamma\lambda}(v_{0\lambda}')v_{0\lambda}'\geq 
    b_1\int_{\Gamma}|J_{\Gamma\lambda} v_{0\lambda}'|^2 + 
    \lambda\int_\Gamma|\alpha_{\Gamma\lambda}(v_{0\lambda}')|^2 - b_2|\Gamma|\,,
  \]
  while on the right hand side, since $v_{0\lambda}'-J_{\Gamma\lambda} v_{0\lambda}'=
  \lambda\alpha_{\Gamma\lambda}(v_{0\lambda}')$, for every $\delta>0$ we have
  \[
    \int_\Gamma w_{0\lambda}v_{0\lambda}' \leq \frac1{2\delta}\int_\Gamma|w_{0\lambda}|^2
  +\frac\delta2\int_\Gamma|v_{0\lambda}'|^2\leq
  \frac1{2\delta}\int_\Gamma|w_{0\lambda}|^2
  +\delta\lambda^2\int_\Gamma|\alpha_{\Gamma\lambda}(v_{0\lambda}')|^2 + 
  \delta\int_\Gamma|J_{\Gamma\lambda} v_{0\lambda}'|^2\,,
  \]
  where $w_{0\lambda}$ is bounded in $H_\Gamma$ uniformly in $\lambda$ by \eqref{g'_bis}--\eqref{u0'}.
  Now, recall that either \eqref{coerc1} or \eqref{u0'_ter} is in order: we distinguish the two cases.
  Under hypothesis \eqref{coerc1}, we have, on the left hand side,
  \[
    \int_\Omega\alpha_\lambda(u_{0\lambda}')u_{0\lambda}'\geq 
    a_1\int_{\Omega}|J_\lambda u_{0\lambda}'|^2 + \lambda\int_\Omega|\alpha_\lambda(u_{0\lambda}')|^2 - a_2|\Omega|
  \]
  while on the right hand side, since $u_{0\lambda}'-J_\lambda u_{0\lambda}'=\lambda\alpha_\lambda(u_{0\lambda}')$,
  \[
  \int_\Omega z_{0\lambda}u_{0\lambda}' \leq \frac1{2\delta}\int_\Omega|z_{0\lambda}|^2
  +\frac\delta2\int_\Omega|u_{0\lambda}'|^2\leq
  \frac1{2\delta}\int_\Omega|z_{0\lambda}|^2
  +\delta\lambda^2\int_\Omega|\alpha_\lambda(u_{0\lambda}')|^2 + \delta\int_\Omega|J_\lambda u_{0\lambda}'|^2\,.
  \]
  Since $z_{0\lambda}$ is uniformly bounded in $H$ by \eqref{g'_bis}--\eqref{u0'_bis}, choosing $\delta>0$ sufficiently small
  and rearranging the terms yields the desired estimate. Otherwise, if \eqref{u0'_ter} is in order, 
  then $z_{0\lambda}$ is uniformly bounded also in $V$ by \eqref{u0'_ter}
  and we can estimate the term on the right hand side in the duality $V$--$V^*$:
  \[
    \int_\Omega z_{0\lambda}u_{0\lambda}'=
  -\lambda\int_\Omega z_{0\lambda}\mu_{0\lambda} - \int_\Omega \nabla z_{0\lambda}\cdot\nabla\mu_{0\lambda}\leq
  \frac\lambda2\int_\Omega|\mu_{0\lambda}|^2 + \frac12\int_\Omega|\nabla\mu_{0\lambda}|^2 + \norm{z_{0\lambda}}_V^2\,,
  \]
  from which the desired estimate follows rearranging the terms. Note that we have used the fact that
  \[
    \widehat{\alpha_\lambda^{-1}}(\alpha_\lambda(u_{0\lambda}'))\leq
    \widehat\alpha_\lambda(u_{0\lambda}')+
    \widehat{\alpha_\lambda^{-1}}(\alpha_\lambda(u_{0\lambda}'))=
    \alpha_\lambda(u_{0\lambda}')u_{0\lambda}'
  \]
  and the equivalent statement for $\alpha_{\Gamma}$.
\end{proof}

We are ready now to perform the estimate. The intuitive idea is to test equation \eqref{eq1_app} by $\partial_t \mu_\lambda$, 
the time-derivative of equation \eqref{eq2_app} by $\partial_t u_\lambda$, take the difference and integrate. However,
the regularity of the approximated solutions does not allow us to do so. Consequently, we prove by hand that the 
resulting estimate holds anyway. To this end, we proceed in a technical way
through a discrete-time argument
as in \cite[Section~5.2]{bcst1}, to which we refer for further detail; however, we avoid
any detailed computation for sake of conciseness.

Fix $t\in[0,T]$ and set, for every $n\in\enne$, $\tau_n:=\frac{t}{n}$ and $t:=i\tau_n$ for $i\in\{0,\ldots,n\}$.
Now, by the regularities given by Lemma~\ref{prop_app2}, 
we note that \eqref{eq1_app}--\eqref{eq2_app} and \eqref{eq3_app} hold for every $s\in[0,T]$. Hence, it makes sense to
test \eqref{eq1_app} at time $t_n^i$ by $\mu_\lambda(t_n^i)-\mu_\lambda(t_n^{i-1})$,
the difference between \eqref{eq2_app} at $t_n^i$ and at $t_n^{i-1}$ by $\partial_t u_\lambda(t_n^i)$,
and take the difference. Moreover, since $\partial \widehat{\alpha_\lambda^{-1}}=\alpha_\lambda^{-1}$, 
for every $i=1,\ldots,n$, we have that 
\[
  \left(\alpha_\lambda(\partial_t u_\lambda(t_n^i))-\alpha_\lambda(\partial_t u_\lambda(t_n^{i-1}))\right)\partial_t u_\lambda(t_n^i)
  \geq\widehat{\alpha_\lambda^{-1}}(\alpha_\lambda(\partial_t u_\lambda(t_n^i)))-
  \widehat{\alpha_\lambda^{-1}}(\alpha_{\Gamma\lambda}(\partial_t u_\lambda(t_n^{i-1})))
\]
and similarly for the terms in $\alpha_{\Gamma\lambda}$. Hence, 
integrating by parts and summing over $i$ yields
(after some technical computations analogue to the ones in \cite[Section~5.2]{bcst1}),
\[\begin{split}
  &\frac\lambda2\int_\Omega|\mu_\lambda(t)| + \frac12\int_\Omega|\nabla\mu_\lambda(t)|^2
  +\frac\lambda2\int_\Omega|\partial_t u_\lambda(t)|^2 
  + \int_\Omega\widehat{\alpha_\lambda^{-1}}(\alpha_\lambda(\partial_t u_\lambda(t)))\\
  &\quad+\lambda\int_{Q_t}|\partial_t u_\lambda|^2+\int_{Q_t}|\nabla\partial_t u_\lambda|^2
  +\int_{Q_t}\beta_\lambda'(u_\lambda)|\partial_t u_\lambda|^2
  +\frac\lambda2\int_\Gamma|\partial_t v_\lambda(t)|^2\\
  &\quad+\int_\Gamma\widehat{\alpha_{\Gamma\lambda}^{-1}}(\alpha_{\Gamma\lambda}(\partial_t v_\lambda(t)))
  +\eps\int_{\Sigma_t}|\nabla_\Gamma\partial_t v_\lambda|^2
  +\int_{\Sigma_t}\beta_{\Gamma\lambda}'(v_\lambda)|\partial_t v_\lambda|^2\\
  &\leq\frac\lambda2\int_\Omega|\mu_{0\lambda}|^2 + \frac12\int_\Omega|\nabla\mu_{0\lambda}|^2
  +\frac\lambda2\int_\Omega|u_{0\lambda}'|^2
  +\!\int_\Omega\widehat{\alpha_\lambda^{-1}}(\alpha_\lambda(u_{0\lambda}'))
  +\frac\lambda2\int_\Gamma|v_{0\lambda}'|^2 
  +\!\int_\Gamma\widehat{\alpha_{\Gamma\lambda}^{-1}}(\alpha_{\Gamma\lambda}(v_{0\lambda}'))\\
  &\quad+\int_{Q_t}\partial_t g_\lambda \partial_t u_\lambda -\int_{Q_t} T_\lambda'(\pi(u_\lambda))\pi'(u_\lambda)|\partial_t u_\lambda|^2
  +\int_{\Sigma_t}\partial_t g_{\Gamma\lambda} \partial_t v_\lambda 
  -\int_{\Sigma_t} T_\lambda'(\pi_\Gamma(v_\lambda))\pi_\Gamma'(v_\lambda)|\partial_t v_\lambda|^2\,.
\end{split}\]
Now, the first six terms and the last term on
the right-hand side are bounded uniformly in $\lambda$ thanks to Lemma~\ref{init_reg}
and the estimate \eqref{est4}, respectively (recall that $|T_\lambda'|\leq 1$ and $|\pi_\Gamma'|\leq C_{\pi_\Gamma}$).
Moreover, the three remaining terms can be estimated using the duality $V$--$V^*$, the assumption
\eqref{g'}, the Young inequality and \eqref{normH} by
\[
  C_\delta+ \delta\norm{\partial_t u_\lambda}^2_{L^2(0,t; V)} \leq C_\delta + \delta\norm{\nabla \partial_t u_\lambda}^2_{L^2(0,t; H)}
  +\delta \lambda^2\norm{\mu_\lambda}^2_{L^2(0,t; H)}\,,
\]
for every $\delta >0$.
Hence, choosing $\delta$ sufficiently small, 
we deduce that there is a positive constant $C$ such that
\begin{gather}
  \label{est14}
  \norm{\nabla\partial_t u_\lambda}_{L^2(0,T; H)} + \lambda^{1/2}\norm{\partial_t u_\lambda}_{L^\infty(0,T; H)}\leq C\,,\\
  \label{est15}
  \norm{\nabla_\Gamma \partial_t v_\lambda}_{L^2(0,T; H_\Gamma)} + 
  \lambda^{1/2}\norm{\partial_t v_\lambda}_{L^\infty(0,T; H_\Gamma)}\leq C\,,\\
  \label{est16}
  \norm{\nabla\mu_\lambda}_{L^\infty(0,T; H)} + \lambda^{1/2}\norm{\mu_\lambda}_{L^\infty(0,T; H)}\leq C\,,\\
  \label{est16bis}
  \norm{\widehat{\alpha_\lambda^{-1}}(\alpha_\lambda(\partial_t u_\lambda))}_{L^\infty(0,T; L^1(\Omega))}
  +\norm{\widehat{\alpha_{\Gamma\lambda}^{-1}}(\alpha_{\Gamma\lambda}(\partial_t v_\lambda))}_{L^\infty(0,T; L^1(\Gamma))}\leq C
\end{gather}
Thanks to \eqref{est14}--\eqref{est15}, condition \eqref{normH}
and \eqref{est1} and \eqref{est5}, it follows that 
$\partial_t u_\lambda$ and $\partial_t v_\lambda$ are uniformly bounded
in $L^2(0,T; V)$ and $L^2(0,T; V_{\Gamma})$, respectively.
Moreover, integrating \eqref{alpha_sub} it easily follows that $\widehat{\alpha}_\lambda$
and $\widehat\alpha_{\Gamma\lambda}$ are uniformly bounded in $\lambda$ from above 
by a quadratic function: hence, $\widehat{\alpha_\lambda^{-1}}=(\widehat\alpha_\lambda)^*$ 
and $\widehat{\alpha_{\Gamma\lambda}^{-1}}=(\widehat\alpha_{\Gamma\lambda})^*$
are uniformly bounded from below by a quadratic function. Consequently, from
the estimate \eqref{est16bis} we infer also that
\beq
  \label{est17}
  \norm{\alpha_\lambda(\partial_t u_\lambda)}_{L^\infty(0,T; H)} + 
  \norm{\alpha_{\Gamma\lambda}(\partial_t v_\lambda)}_{L^\infty(0,T; H_\Gamma)}\leq C\,.
\eeq
Moreover, from the coercivity of $\alpha_\Gamma$ and the 
Young inequality, we have 
\[
  b_1|J_{\Gamma\lambda}\partial_t v_\lambda|^2 - b_2 \leq 
  \alpha_{\Gamma\lambda}(\partial_t v_\lambda)J_{\Gamma\lambda}\partial_t v_\lambda\leq
  \frac{b_1}{2}|J_{\Gamma\lambda}\partial_t v_\lambda|^2
  +\frac{1}{2b_1}|\alpha_{\Gamma\lambda}(\partial_t v_\lambda)|^2,
\]
so that by \eqref{est17} we deduce that 
\beq
  \label{est17_bis}
  \norm{J_{\Gamma\lambda}\partial_t v_\lambda}_{L^\infty(0,T; H)}\leq C\,.
\eeq
Finally, arguing exactly as in Section~\ref{second}
but using the stronger estimates \eqref{est14}--\eqref{est17},
it is readily seen that $(\mu_\lambda)_\Omega$ is uniformly bounded in $L^\infty(0,T)$, so that by \eqref{est16} we have
\[
  \norm{\mu_\lambda}_{L^\infty(0,T; V)}\leq C\,.
\]
Moreover, thanks to \eqref{est14} and \eqref{est16}, by comparison in \eqref{eq1_app}
and elliptic regularity we have 
\beq\label{est18}
  \norm{\mu_\lambda}_{L^\infty(0,T; V)\cap L^2(0,T; W_{\bf n}\cap H^3(\Omega))}
  +\norm{\partial_t u_\lambda}_{L^\infty(0,T; V^*)}\leq C\,.
\eeq
It is clear that under the assumption \eqref{coerc1}, the same argument 
ensures that $J_\lambda\partial_t u_\lambda$ is uniformly bounded in $L^{\infty}(0,T; H)$ as well,
hence also $\mu_\lambda$ in $L^\infty(0,T; W_{\bf n})$ form \eqref{eq1_app},
from which the last sentence of Theorem~\ref{thm2} follows.

\subsection{The third estimate}\label{third'}
For every $t\in[0,T]$, we test equation \eqref{eq2_app} by $-\Delta u_\lambda(t)$ and integrate by parts:
\[\begin{split}
  \int_\Omega|\Delta u_\lambda(t)|^2 &+ 
  \int_\Omega\beta_\lambda'(u_\lambda(t))|\nabla u_\lambda(t)|^2
  +\eps\int_\Gamma\beta_\lambda'(v_\lambda(t))|\nabla_\Gamma v_\lambda(t)|^2
  +\int_\Gamma\beta_\lambda(v_\lambda(t))\beta_{\Gamma\lambda}(v_\lambda(t))\\
  &=-\int_\Omega\left(g_\lambda(t)-T_\lambda\pi(u_\lambda(t))-\lambda u_\lambda(t)+\mu_\lambda(t)-\lambda\partial_t u_\lambda(t)
  -\alpha_\lambda(\partial_t u_\lambda(t))\right)\Delta u_\lambda(t)\\
  &\quad-\int_\Gamma\left(g_{\Gamma\lambda}(t)-T_\lambda\pi_\Gamma(v_\lambda(t))
  -\lambda\partial_t v_\lambda(t)-\alpha_{\Gamma\lambda}(\partial_t v_\lambda(t))\right)\beta_\lambda(v_\lambda(t))\,.
\end{split}\]
Thanks to \eqref{g} the estimates \eqref{est14}--\eqref{est17}, the terms in brackets on the right hand side
are bounded uniformly in $\lambda$. Hence, using the weighted Young inequality and 
the hypothesis \eqref{dom1} as in Section~\ref{third}, we infer that 
\beq
  \label{est19}
  \norm{\Delta u_\lambda}_{L^\infty(0,T; H)} + \norm{\beta_\lambda(u_\lambda)}_{L^\infty(0,T; H_\Gamma)}\leq C\,.
\eeq
By comparison in \eqref{eq2_app} we deduce that 
\beq
  \label{est20}
  \norm{\beta_\lambda(u_\lambda)}_{L^\infty(0,T; H)}\leq C\,.
\eeq
Moreover, by the classical results on elliptic regularity (see \cite[Thm.~3.2]{brezzi-gilardi}), 
estimate \eqref{est19} implies, together with \eqref{est2} and \eqref{est1'}, that
\beq
  \label{est21}
  \eps^{1/2}\norm{u_\lambda}_{L^\infty(0,T; H^{3/2}(\Omega))}+\eps^{1/2}\norm{\partial_{\bf n}u_\lambda}_{L^\infty(0,T; H_\Gamma)}
  \leq C\,
\eeq
and, by comparison in \eqref{eq3_app}, 
\[
  \norm{-\eps^{3/2}\Delta_\Gamma v_\lambda + \eps^{1/2}\beta_{\Gamma\lambda}(v_\lambda)}_{L^\infty(0,T; H_\Gamma)}\leq C\,.
\]
We deduce, as usual, that 
\beq
  \label{est22}
  \eps^{3/2}\norm{\Delta_\Gamma v_\lambda}_{L^\infty(0,T; H_\Gamma)}
  +\eps^{1/2}\norm{\beta_{\Gamma\lambda(v_\lambda)}}_{L^\infty(0,T; H_\Gamma)}\leq C\,.
\eeq

\subsection{The passage to the limit}
Taking into account \eqref{est1}--\eqref{est2}, \eqref{est4}--\eqref{est1'} and \eqref{est14}--\eqref{est22}, 
recalling that $\eps>0$ is fixed, we infer that there are
\begin{gather*}
  u \in W^{1,\infty}(0,T; V^*)\cap H^1(0,T; V)\cap L^\infty(0,T; W)\,, \\
  v \in W^{1,\infty}(0,T; H)\cap H^1(0,T; V_\Gamma)\cap L^\infty(0,T; W_\Gamma)\,,\\
  \mu \in L^\infty(0,T; W_{\bf n})\cap L^2(0,T; H^3(\Omega))\,,\\
  \eta\,,\xi \in L^\infty(0,T; H)\,, \qquad \eta_\Gamma\,,\xi_\Gamma \in L^\infty(0,T; H_\Gamma)\,,
\end{gather*}
such that, along a subsequence that we still denote by $\lambda$ for simplicity,
\begin{align*}
  u_\lambda \wstarto u \quad\text{in } W^{1,\infty}(0,T; V^*)\cap L^\infty(0,T; W)\,,& \qquad
  u_\lambda \wto v \quad\text{in } H^1(0,T; V)\,,\\
  v_\lambda \wstarto u \quad\text{in } W^{1,\infty}(0,T; H_\Gamma)\cap L^\infty(0,T; W_{\Gamma})\,,&
  \qquad v_\lambda \wto v \quad\text{in } H^1(0,T; V_\Gamma)\,,\\
  \mu_\lambda \wstarto \mu \quad\text{in } L^\infty(0,T; V)\,,&
  \qquad \mu_\lambda \wto \mu \quad\text{in } L^2(0,T; W_{\bf n}\cap H^3(\Omega))\,,\\
  \alpha_\lambda(\partial_t u_\lambda) \wstarto \eta \quad\text{in } L^\infty(0,T; H)\,,& \qquad
  \alpha_{\Gamma\lambda}(\partial_t v_\lambda) \wstarto \eta_\Gamma \quad\text{in } L^\infty(0,T; H_\Gamma)\,,\\
  \beta_\lambda(u_\lambda) \wstarto \xi \quad\text{in } L^\infty(0,T; H)\,,& \qquad
  \beta_{\Gamma\lambda}(v_\lambda) \wstarto \xi_\Gamma \quad\text{in } L^\infty(0,T; H_\Gamma)
\end{align*}
and
\[
  \lambda u_\lambda \to 0 \quad\text{in } W^{1,\infty}(0,T; H)\,, \quad
  \lambda v_\lambda \to 0 \quad\text{in } W^{1,\infty}(0,T; H_\Gamma)\,, \quad
  \lambda\mu_\lambda\to0 \quad\text{in } L^\infty(0,T; H)\,.
\]
At this point, it is straightforward to conclude as in Section~\ref{limit}
and Theorem~\ref{thm2} is proved.


\section{The third existence result}\label{proof3}
First of all, note that all the estimates which do not involve the assumption \eqref{alpha_sub}
continue to hold also in this setting. Namely, going back to Sections~\ref{first} and \ref{first'},
it is readily seen that \eqref{est1}--\eqref{est2}, \eqref{est4}--\eqref{est1'}, \eqref{est14}--\eqref{est16bis}
are satisfied.

Secondly, by \eqref{ip_0dom}, there is $\delta>0$ such that $\pm\delta\in D(\alpha)\cap D(\alpha_\Gamma)$. Hence,
by the Young inequality we have
\begin{gather*}
  \pm\delta\alpha_\lambda(\partial_t u_\lambda)\leq \widehat\alpha_\lambda(\pm \delta) 
  + \widehat{\alpha_{\lambda}^{-1}}(\partial_t u_\lambda)
  \leq \widehat\alpha(\pm\delta) + \widehat{\alpha_{\lambda}^{-1}}(\partial_t u_\lambda)\,,\\
  \pm\delta\alpha_{\Gamma\lambda}(\partial_t v_\lambda)\leq \widehat\alpha_{\Gamma\lambda}(\pm \delta) 
  + \widehat{\alpha_{\Gamma\lambda}^{-1}}(\partial_t v_\lambda)
  \leq \widehat\alpha_{\Gamma}(\pm\delta) + \widehat{\alpha_{\Gamma\lambda}^{-1}}(\partial_t v_\lambda)\,,
\end{gather*}
so that by \eqref{est16bis} we deduce that 
\[
  \norm{\alpha_\lambda(\partial_t u_\lambda)}_{L^\infty(0,T; L^1(\Omega))}+
  \norm{\alpha_{\Gamma\lambda}(\partial_t v_\lambda)}_{L^\infty(0,T; L^1(\Gamma))}\leq C\,.
\]
Furthermore, thanks to the assumptions \eqref{ip_beta}--\eqref{ip_beta_g}, 
the estimates \eqref{est2} and \eqref{est1'}, as well as the continuous inclusions
$V\embed L^6(\Omega)$ and $V_\Gamma\embed L^q(\Gamma)$ (for every $q\geq1$), we have that 
for every $q\in[1,+\infty)$
\beq\label{est23'}
  \norm{\beta_\lambda(u_\lambda)}_{L^\infty(0,T; L^{6/5}(\Omega))} +
  \norm{\beta_{\Gamma\lambda}(v_\lambda)}_{L^\infty(0,T; L^{q}(\Gamma))}
  \leq C
\eeq
for every $q\geq1$.
Consequently, testing \eqref{eq2_app} by the constants $\pm1$ we get
\[\begin{split}
  \pm|\Omega|(\mu_\lambda(t))_\Omega&\leq
  \int_\Omega|\lambda\partial_t u_\lambda + \alpha_\lambda(\partial_t u_\lambda) + \lambda u_\lambda
  +\beta_\lambda(u_\lambda)+T_\lambda\pi(u_\lambda)|(t) + |\Omega||(g_\lambda(t))_\Omega|\\
  &+ \int_\Gamma|\lambda\partial_t v_\lambda + \alpha_{\Gamma\lambda}(\partial_t v_\lambda)
  +\beta_{\Gamma\lambda}(v_\lambda)+T_\lambda\pi_\Gamma(v_\lambda)|(t) + |\Gamma||(g_{\Gamma\lambda}(t))_\Gamma|\,,
\end{split}\]
where the right-hand side is bounded in $L^\infty(0,T)$ thanks to the estimates already shown, 
\eqref{est14}--\eqref{est15} and by assumption \eqref{g'}. We infer together with \eqref{est16} that
\[
  \norm{\mu_\lambda}_{L^\infty(0,T; V)}\leq C\,.
\]
Furthermore, by comparison in \eqref{eq1_app} and the estimates \eqref{est14}--\eqref{est16} that 
\beq
  \label{est23}
  \norm{\mu_\lambda}_{L^\infty(0,T; V)\cap L^2(0,T; W_{\bf n}\cap H^3(\Omega))}
  + \norm{\partial_t u_\lambda}_{L^\infty(0,T; V^*)}\leq C\,.
\eeq
Again, if also \eqref{coerc1} holds, the same argument 
ensures that $J_\lambda\partial_t u_\lambda$ is uniformly bounded in $L^{\infty}(0,T; H)$,
hence also $\mu_\lambda$ in $L^\infty(0,T; W_{\bf n})$ form \eqref{eq1_app},
from which the last sentence of Theorem~\ref{thm3} follows.

Let us focus now on the main estimate. We know that the approximated 
problem can be written as
\[
  A_\lambda(\partial_t u_\lambda, \partial_t v_\lambda) + B_\lambda(u_\lambda,v_\lambda)=
  (g_\lambda,g_{\Gamma\lambda})-(T_\lambda\pi(u_\lambda), T_\lambda\pi_\Gamma(v_\lambda))\,,
\]
where the operators $A_\lambda$ and $B_\lambda$ have been introduced in Section~\ref{approx}.
Note that by \eqref{est2} and \eqref{est1'}, we have that $(u_\lambda, v_\lambda)_\lambda$ is bounded
in $L^\infty(0,T; \V)$: hence, by linearity and boundedness of the operator
\[
  (-\Delta, \partial_{\bf n} - \eps\Delta_\Gamma): \V\to \V^*\,,
\] 
we deduce that $(-\Delta u_\lambda, \partial_{\bf n}u_\lambda - \eps\Delta v_\lambda)_\lambda$
is bounded uniformly in $L^\infty(0,T; \V^*)$. Moreover, since $L^{6/5}(\Omega)\embed V^*$ and
$L^{q'}(\Gamma)\embed V_\Gamma^*$ for every $q'\in(1,+\infty]$, by \eqref{est23'} we deduce that 
$(\beta_\lambda(u_\lambda), \beta_{\Gamma\lambda}(v_\lambda))_\lambda$ is bounded in 
$L^\infty(0,T; \V^*)$ as well. Hence, we infer that 
\[
  \norm{B_\lambda(u_\lambda, v_\lambda)}_{L^\infty(0,T; \V^*)}\leq C\,.
\]
By comparison in the equation written above we have then
\[
  \norm{(\alpha_\lambda(\partial_t u_\lambda), \alpha_{\Gamma\lambda}(\partial_t v_\lambda))_\lambda}_{L^\infty(0,T; \V^*)}\leq C\,.
\]

Let us pass to the limit. The estimates that we have collected ensure that
there are
\begin{gather*}
  u \in W^{1,\infty}(0,T; V^*)\cap H^1(0,T; V)\cap L^\infty(0,T; W)\,, \\
  v \in W^{1,\infty}(0,T; H)\cap H^1(0,T; V_\Gamma)\cap L^\infty(0,T; W_\Gamma)\,,\\
  \mu \in L^\infty(0,T; W_{\bf n})\cap L^2(0,T; H^3(\Omega))\,,\\
  \xi \in L^\infty(0,T; L^{6/5}(\Omega))\,, \qquad \xi_\Gamma \in L^\infty(0,T; L^q(\Gamma)) \quad\forall\,q\in[1,+\infty)\,,\\
  \eta_w \in L^\infty(0,T; \V^*)\,,
\end{gather*}
such that, along a subsequence that we still denote by $\lambda$ for simplicity,
\begin{align*}
  u_\lambda \wstarto u \quad\text{in } W^{1,\infty}(0,T; V^*)\cap L^\infty(0,T; W)\,,& \qquad
  u_\lambda \wto v \quad\text{in } H^1(0,T; V)\,,\\
  v_\lambda \wstarto u \quad\text{in } W^{1,\infty}(0,T; H_\Gamma)\cap L^\infty(0,T; W_{\Gamma})\,,&
  \qquad v_\lambda \wto v \quad\text{in } H^1(0,T; V_\Gamma)\,,\\
  \mu_\lambda \wstarto \mu \quad\text{in } L^\infty(0,T; V)\,,&
  \qquad \mu_\lambda \wto \mu \quad\text{in } L^2(0,T; W_{\bf n}\cap H^3(\Omega))\,,\\
  \beta_\lambda(u_\lambda) \wstarto \xi \quad\text{in } L^\infty(0,T; L^{6/5}(\Omega))\,,& \qquad
  \beta_{\Gamma\lambda}(v_\lambda) \wstarto \xi_\Gamma \quad\text{in } L^\infty(0,T; H_\Gamma)\,,\\
  (\alpha_\lambda(\partial_t u_\lambda), \alpha_{\Gamma\lambda}(\partial_t v_\lambda))\wstarto\eta_w&
  \qquad\text{in } L^\infty(0,T; \V^*)
\end{align*}
and
\[
  \lambda u_\lambda \to 0 \quad\text{in } W^{1,\infty}(0,T; H)\,, \quad
  \lambda v_\lambda \to 0 \quad\text{in } W^{1,\infty}(0,T; H_\Gamma)\,, \quad
  \lambda\mu_\lambda\to0 \quad\text{in } L^\infty(0,T; H)\,.
\]
If the stronger condition \eqref{ip_beta'} is in order, then the continuous embedding $V\embed L^6(\Omega)$
and \eqref{est1'} imply that $(\beta_\lambda(u_\lambda))_\lambda$ is bounded in $L^\infty(0,T; H)$, from 
which $\xi \in L^\infty(0,T; H)$ as well.
Testing the approximated equations \eqref{eq2_app}--\eqref{eq3_app} by a generic element $(\varphi, \psi)\in \V$,
integrating by parts and letting $\lambda\to0^+$, it is a standard matter to check that
\[\begin{split}
    \int_\Omega\mu(t)\varphi&=\ip{\eta_w(t)}{(\varphi,\psi)}_\V + \int_\Omega\nabla u(t)\cdot\nabla\varphi + 
    \int_\Omega\left(\xi(t)+\pi(u(t))-g(t)\right)\varphi\\
    &+ \eps\int_\Gamma\nabla_\Gamma v(t)\cdot\nabla_\Gamma\psi
    +\int_\Gamma(\xi_\Gamma(t)+\pi_\Gamma(v(t))-g_\Gamma(t))\psi\,.
\end{split}\]
Moreover, proceeding as in the previous sections, we also have $\xi\in\beta(u)$ a.e.~in $Q$ and $\xi_\Gamma\in \beta_\Gamma(v)$
a.e.~in $\Sigma$. Finally, as in Section~\ref{limit}, comparing
the approximated equations \eqref{eq2_app}--\eqref{eq3_app} and the corresponding limit ones, we can infer that 
\[
   \limsup_{\lambda\searrow0}\left[\int_Q\alpha_\lambda(\partial_t u_\lambda)\partial_t u_\lambda
   +\int_\Sigma\alpha_{\Gamma\lambda}(\partial_t v_\lambda)\partial_t v_\lambda\right]\leq
   \int_0^T\ip{\eta_w(t)}{(\partial_t u(t), \partial_t v(t))}_\V\,dt\,,
\]
which implies by a well-known criterion on maximal monotonicity that 
$\eta_w \in \widetilde\alpha_w(\partial_t u, \partial_t v)$.


\section{The uniqueness result}
\setcounter{equation}{0}
\label{proof4}

In the hypotheses \eqref{ip_uniq}--\eqref{ip_uniq'} of Theorem~\ref{thm4}, we clearly have
$\xi_i=F'(u_i)-\pi(u_i)$ and $\xi_{\Gamma i}=F_\Gamma'(v_i)-\pi_\Gamma(v_i)$
for $i=1,2$. Now,
we write the difference of the equations
\eqref{1}--\eqref{3} at $i=1$ and $i=2$, test \eqref{1} by $\mu_1-\mu_2$, \eqref{2} by $-\partial_t(u_1-u_2)$
and sum: by standard computations, the monotonicity of $\alpha$ and \eqref{ip_uniq'} we obtain
\[\begin{split}
  \int_{Q_t}&|\nabla(\mu_1-\mu_2)|^2
  +\frac12\int_\Omega|\nabla (u_1-u_2)(t)|^2 +\frac\eps2\int_\Gamma|\nabla_\Gamma(v_1-v_2)(t)|^2
  +\widetilde{b_1}\int_{\Sigma_t}|\partial_t(v_1-v_2)|^2\\
  &\qquad+ \int_{Q_t}\left(F'(u_1)-F'(u_2)\right)\partial_t(u_1-u_2)
   + \int_{\Sigma_t}(F_\Gamma'(v_1)-F'_\Gamma(v_2))\partial_t(v_1-v_2)\leq0
\end{split}\]
for every $t\in[0,T]$. We are now inspired by the argument contained in
the works \cite[Thm.~2.2]{ef-zel} and \cite[p.~689]{mir-sch}: note that 
\[\begin{split}
  &\left(F'(u_1)-F'(u_2)\right)\partial_t(u_1-u_2) \\
  &\qquad= 
  \partial_t\left[F(u_1)-F(u_2)-F'(u_2)(u_1-u_2)\right]
  -\left[F'(u_1)-F'(u_2)-F''(u_2)(u_1-u_2)\right]\partial_t u_2
\end{split}\]
and similarly
\[\begin{split}
  &\left(F'_\Gamma(v_1)-F_\Gamma'(v_2)\right)\partial_t(v_1-v_2) \\
  &\qquad= \partial_t\left[F_\Gamma(v_1)-F_\Gamma(v_2)-F'_\Gamma(v_2)(v_1-v_2)\right]
  -\left[F_\Gamma'(v_1)-F_\Gamma'(v_2)-F''_\Gamma(v_2)(v_1-v_2)\right]\partial_t v_2\,,
\end{split}\]
so that 
\[\begin{split}
  &\int_{Q_t}|\nabla(\mu_1-\mu_2)|^2
  +\frac12\int_\Omega|\nabla (u_1-u_2)(t)|^2 +\frac\eps2\int_\Gamma|\nabla_\Gamma(v_1-v_2)(t)|^2
  +\widetilde{b_1}\int_{\Sigma_t}|\partial_t(v_1-v_2)|^2\\
  &\qquad+ \int_\Omega\left[F(u_1)-F(u_2)-F'(u_2)(u_1-u_2)\right](t)
   + \int_\Gamma\left[F_\Gamma(v_1)-F_\Gamma(v_2)-F'_\Gamma(v_2)(v_1-v_2)\right](t) \\
  &\leq\int_{Q_t}\left[F'(u_1)-F'(u_2)-F''(u_2)(u_1-u_2)\right]\partial_t u_2
  +\int_{\Sigma_t}\left[F_\Gamma'(v_1)-F_\Gamma'(v_2)-F''_\Gamma(v_2)(v_1-v_2)\right]\partial_t v_2
\end{split}\]
for every $t\in[0,T]$. Now, by the mean value theorem it is readily seen that
\begin{align*}
  F(u_1)-F(u_2)-F'(u_2)(u_1-u_2)&\geq -C_\pi|u_1-u_2|^2\,, \\
  F_\Gamma(v_1)-F_\Gamma(v_2)-F'_\Gamma(v_2)(v_1-v_2)&\geq
  -C_{\pi_\Gamma}|v_1-v_2|^2\,,
\end{align*}
while the usual Taylor expansion for $F'$ yields
\[
\left[F'(u_1)-F'(u_2)-F''(u_2)(u_1-u_2)\right]\partial_t u_2=
 \frac12F'''(\tilde{u}_{12})|u_1-u_2|^2\partial_t u_2
\]
for a certain $\tilde{u}_{12}$ between $u_1$ and $u_2$. Now, recall that
$\partial_t u_2 \in L^2(0,T; V)\embed L^2(0,T; L^6(\Omega))$ and 
$u_i \in L^\infty(0,T; W)\embed L^\infty(Q)$ for $i=1,2$: this implies in particular
that $F'''(\tilde u_{12})\in L^\infty(Q)$, 
because $F''' \in L^\infty_{loc}(\erre)$ by \eqref{ip_uniq}. Hence, recalling also that 
$u_1-u_2$ has null mean, we have that 
\[\begin{split}
  \int_{Q_t}F'''(\tilde{u}_{12})|u_1-u_2|^2\partial_t u_2&\leq
  \norm{F'''(\tilde{u}_{12})}_{L^\infty(Q)}\int_0^t\norm{\partial_t u_2(s)}_{L^6(\Omega)}\norm{|u_1-u_2|^2(s)}_{L^{6/5}(\Omega)}\,ds\\
  &\leq C\int_0^t\norm{\partial_t u_2(s)}_V\norm{\nabla(u_1-u_2)(s)}^2_H\,ds
\end{split}\]
for a certain constant $C>0$. Similarly,  we obtain
\[
  \int_{\Sigma_t}\left[F_\Gamma'(v_1)-F_\Gamma'(v_2)-F''_\Gamma(v_2)(v_1-v_2)\right]\partial_t v_2
  \leq C\int_0^t\norm{\partial_t v_2(s)}_{V_\Gamma}\norm{\nabla_\Gamma(v_1-v_2)(s)}^2_{H_\Gamma}\,ds\,.
\]
Furthermore, by the Young inequality we can write (updating the constant $C$ at each step)
\[\begin{split}
  C_\pi\int_\Omega|u_1-u_2|^2(t)&=2C_\pi\int_{Q_t}\partial_t(u_1-u_2)(u_1-u_2)\\
  &\leq
  \frac12\norm{\partial_t(u_1-u_2)}^2_{L^2(0,t; V^*)} +
   C\norm{u_1-u_2}^2_{L^2(0,t; V)}\\
   &\leq\frac12\int_{Q_t}|\nabla(\mu_1-\mu_2)|^2 + C\norm{\nabla(u_1-u_2)}^2_{L^2(0,t; H)}
\end{split}\]
and similarly
\[
C_{\pi_\Gamma} \int_\Gamma|v_1-v_2|^2(t)\leq
\frac{\widetilde{b_1}}2\int_{Q_t}|\partial_t(v_1-v_2)|^2+
  C\norm{\nabla_\Gamma(v_1-v_2)}^2_{L^2(0,t; H_\Gamma)}\,.
\]
Taking into account this information and rearranging the terms yields
\[\begin{split}
  &\frac12\int_{Q_t}|\nabla(\mu_1-\mu_2)|^2
  +\frac12\int_\Omega|\nabla (u_1-u_2)(t)|^2 +\frac\eps2\int_\Gamma|\nabla_\Gamma(v_1-v_2)(t)|^2
  +\frac{\widetilde{b_1}}2\int_{\Sigma_t}|\partial_t(v_1-v_2)|^2\\
  &\qquad\leq C\int_0^t(1+\norm{\partial_t u_2(s)}_{V})\norm{\nabla(u_1-u_2)(s)}^2_{H}\,ds\\
  &\qquad\quad+
  C\int_0^t(1+\norm{\partial_t v_2(s)}_{V_\Gamma})\norm{\nabla_\Gamma(v_1-v_2)(s)}^2_{H_\Gamma}\,ds
  \qquad\forall\,t\in[0,T]\,,
\end{split}\]
and the thesis follows from the Gronwall lemma.

In order to prove the second part of the theorem, we proceed in exactly the same way:
we test \eqref{1} by $\mu_1-\mu_2$, \eqref{eq_var} by -$(\partial_t (u_1-u_2), \partial_t(v_1-v_2))\in\V$
and sum.
The only difference here is that the estimate on the term involving $F'''$ has to performed
using the weaker regularity of the solutions and the hypothesis \eqref{ip_uniq''}, 
together with the fact that $V\embed L^6(\Omega)$, as follows:
\[\begin{split}
  &\int_{Q_t}F'''(\tilde{u}_{12})(u_1-u_2)\partial_t u_2\\
  &\qquad\leq
  M\norm{|Q|+|u_1|^3+|u_2|^3}_{L^\infty(0,T; H)}\int_0^t\norm{\partial_t u_2(s)}_{L^6(\Omega)}\norm{|u_1-u_2|^2(s)}_{L^{3}(\Omega)}\,ds\\
  &\qquad\leq C\left(1+\norm{u_1}^2_{L^\infty(0,T; V)}+\norm{u_2}^2_{L^\infty(0,T; V)}\right)
  \int_0^t\norm{\partial_t u_2(s)}_{V}\norm{\nabla(u_1-u_2)(s)}^2_{H}\,ds
\end{split}\]
Similarly, the term involving $F_\Gamma'''$ is handled using \eqref{ip_uniq'''}
and the inclusion $V_\Gamma\embed L^q(\Gamma)$ for every $q\in[1,+\infty)$.


\section{The asymptotic as $\eps\searrow0$}
\setcounter{equation}{0}
\label{proof5}

For every $\eps>0$, the septuple 
$(u_\eps,v_\eps,\mu_\eps,\eta_\eps,\xi_\eps,\eta_{\Gamma\eps},\xi_{\Gamma\eps})$
is the solution satisfying \eqref{u}--\eqref{3} given by Theorem~\ref{thm1}.
Hence, recalling how 
such solutions were built from the approximated ones, 
all the estimates that we performed in Section~\ref{proof1}
(and that are $\eps$-independent)
are preserved.
In particular, going back to Section~\ref{proof1}
and taking \eqref{est_init} into account, it is readily seen that
\begin{gather*}
  \norm{u_\eps}_{L^\infty(0,T; V)\cap H^1(0,T; H)} 
  + \norm{v_\eps}_{L^\infty(0,T; H^{1/2}(\Gamma))\cap H^1(0,T; H_\Gamma)} +
  \eps^{1/2}\norm{v_\eps}_{L^\infty(0,T; V_\Gamma)}\leq C\,,\\
  \norm{\mu_\eps}_{L^2(0,T; W_{\bf n})}\leq C\,,\\
   \norm{\eta_\eps}_{L^2(0,T; H)} + \norm{\eta_{\Gamma\eps}}_{L^2(0,T; H_\Gamma)}+
   \norm{\xi_\eps}_{L^2(0,T; H)} \leq C\,,\\
   \norm{\Delta u_\eps}_{L^2(0,T; H)} + 
   \norm{\partial_{\bf n}u_\eps - \eps\Delta_\Gamma v_\eps + \xi_{\Gamma\eps}}_{L^2(0,T; H_\Gamma)}\leq C\,.
\end{gather*}
By the classical results on elliptic regularity, we can only infer that 
\[
  \norm{\partial_{\bf n}u_\eps}_{L^2(0,T; H^{-1/2}(\Gamma))} + \eps^{1/2}\norm{\partial_{\bf n}u_\eps}_{L^2(0,T; H_\Gamma)}\leq c\,.
\]
Taking into account that $-\Delta_\Gamma:V_\Gamma\to V_\Gamma^*$ is continuous and monotone, we also have that 
\[
  \eps^{1/2}\norm{\Delta_\Gamma v_\eps}_{L^\infty(0,T; V_\Gamma^*)} + \eps^{3/2}\norm{\Delta_\Gamma v_\eps}_{L^2(0,T; H_\Gamma)}\leq c\,,
\]
which yields by interpolation
\[
  \eps\norm{\Delta v_\eps}_{L^2(0,T; H^{-1/2}(\Gamma))}\leq c\,,
\]
hence also, by comparison,
\[
  \norm{\xi_{\Gamma\eps}}_{L^2(0,T; H^{-1/2}(\Gamma))}\leq c\,.
\]
It readily seen that, along a subsequence $(\eps_n)_n$, the weak convergences
of Theorem~\ref{thm5} hold. Furthermore, by the classical compactness results
\cite[Cor.~4, p.~85]{simon} we also have
\[
  u_{\eps_n}\to u \quad\text{in } C^0([0,T]; H)\,, \qquad
  v_{\eps_n}\to v \quad\text{in } C^0([0,T]; H_\Gamma)\,,
\]
which yields $\xi \in \beta(u)$ a.e.~in $Q$ by the strong-weak closure of $\beta$.
Passing to the weak limit as $n\to\infty$ in \eqref{1}--\eqref{3} we deduce that 
$(u,v,\mu,\eta,\xi,\eta_\Gamma,\xi_\Gamma)$ satisfies the limit equations
stated in Theorem~\ref{thm5}.
Moreover, testing \eqref{1} by $\mu_\eps$, \eqref{2} by $-\partial_t u_\eps$ and summing we get
\[\begin{split}
  &\int_Q|\nabla\mu_\eps|^2+\int_{Q}\eta_\eps\partial_t u_\eps +
  \frac12\int_\Omega|\nabla u_\eps(T)|^2+\int_\Omega\widehat\beta(u_\eps(T))\\
   &\qquad+\int_\Sigma\eta_{\Gamma\eps}\partial_t v_\eps + \frac\eps2\int_\Gamma|\nabla_\Gamma v_\eps(T)|^2
   +\int_\Gamma\widehat\beta_\Gamma(v_\eps(T))
   =\frac12\int_\Omega|\nabla u_0^\eps|^2+\frac\eps2\int_\Gamma|\nabla_\Gamma u_0^\eps|^2\\
   &\qquad+ \int_\Omega\widehat\beta(u_0^\eps) + \int_\Gamma\widehat\beta_\Gamma(u_0^\eps)
  +\int_Q(g -\pi(u_\eps))\partial_t u_\eps + \int_\Sigma(g_\Gamma-\pi_\Gamma(v_\eps))\partial_t v_\eps\,,
\end{split}\]
from which, by standard weak lower semicontinuity results, the convergence $u_0^\eps\to u_0$ in $V$ and
the estimate \eqref{est_init},
\[\begin{split}
  \limsup_{n\to\infty}\left(\int_{Q}\eta_\eps\partial_t u_\eps+\int_\Sigma\eta_{\Gamma\eps}\partial_t v_\eps\right)&\leq
  \frac12\int_\Omega|\nabla u_0|^2+ \int_\Omega\widehat\beta(u_0) + \int_\Gamma\widehat\beta_\Gamma(u_0)\\
  &-\int_Q|\nabla\mu|^2-\frac12\int_\Omega|\nabla u(T)|^2-\int_\Omega\widehat\beta(u(T)) - \int_\Gamma\widehat\beta_\Gamma(v(T))\\
  &+\int_Q(g -\pi(u))\partial_t u + \int_\Sigma(g_\Gamma-\pi_\Gamma(v))\partial_t v\,.
\end{split}\]
Now, performing the analogue estimate on the limiting equations, we easily deduce that the right-hand side
coincides with
\[
  \int_Q\eta\partial_t u + \int_\Sigma\eta_\Gamma\partial_t v\,.
\]
Hence, we also have that
$\eta\in\alpha(\partial_t u)$ a.e.~in $Q$ and $\eta_\Gamma\in\alpha_\Gamma(\partial_t v)$
a.e.~in $\Sigma$.
It remains to prove that $\xi_\Gamma\in\beta_{\Gamma w}(v)$ a.e.~in $(0,T)$. To this end, we test
\eqref{1} by $\mathcal N(u_\eps-(u_0^\eps)_\Omega)$, $\eqref{2}$ by $-(u_\eps-(u_0^\eps)_\Omega)$, and sum:
\[\begin{split}
  &\norm{\nabla\mathcal{N}(u_\eps(T)-(u_0^\eps)_\Omega)}_H^2 + 
  \int_{Q}\eta_\eps(u_\eps-(u^\eps_0)_\Omega) + \int_{Q}|\nabla u_\eps|^2 + \int_Q\xi_\eps(u_\eps-(u^\eps_0)_\Omega)\\
  &\qquad+\int_{\Sigma}\eta_{\Gamma\eps}(v_\eps-(u^\eps_0)_\Omega) + \eps\int_\Sigma|\nabla_\Gamma v_\eps|^2
  +\int_\Sigma \xi_{\Gamma\eps}(v_\eps-(u^\eps_0)_\Omega)\\
  &\qquad=\int_Q(g-\pi(u_\eps))(u_\eps-(u^\eps_0)_\Omega) + \int_\Sigma(g_\Gamma-\pi_\Gamma(v_\eps))(v_\eps-(u^\eps_0)_\Omega)\,.
\end{split}\]
Now, recalling that $u_\eps-(u_0^\eps)_\Omega$ has null mean and that $u_\eps\to u$
in $C^0([0,T]; H)$, we have in particular that $u_\eps(T)-(u_0^\eps)_\Omega \to u(T)-(u_0)_\Omega$ in $V^*$,
hence also, by the properties of $\mathcal N$, $\mathcal N(u(T)_\eps-(u_0)_\Omega)\to\mathcal N(u(T)-(u_0)_\Omega)$ in $V$.
Furthermore, using the convergences already proved
and the weak lower semicontinuity of the norms, we infer that
\[\begin{split}
  &\limsup_{\eps\searrow0}\int_\Sigma\xi_{\Gamma\eps} v_\eps \leq
  \int_Q(g-\pi(u))(u-(u_0)_\Omega) + \int_\Sigma(g_\Gamma-\pi_\Gamma(v))(v-(u_0)_\Omega)
  -\int_Q\eta(u-(u_0)_\Omega)\\
  &+\int_\Sigma\xi_{\Gamma}(u_0)_\Omega-\norm{\nabla\mathcal{N}(u(T)-(u_0)_\Omega)}_H^2 - \int_Q|\nabla u|^2
  -\int_Q\xi(u-(u_0)_\Omega) - \int_\Sigma\eta_{\Gamma}(v-(u_0)_\Omega)\,.
\end{split}\]
As before, performing the same estimate on the limiting equations, we
see that the right-hand side coincides with
\[
  \int_0^T\ip{\xi_\Gamma(t)}{v(t)}_{H^{1/2}(\Gamma)}\,dt\,,
\]
and we can conclude by the maximal  monotonicity of $\beta_{\Gamma w}$.

Finally, if the further assumptions \eqref{g'}--\eqref{u0'_ter} hold and $(\eps u_{0|\Gamma}^\eps)_\eps$
is bounded in $W_\Gamma$, we can proceed similarly performing the estimates in
Section~\ref{proof2} instead. In particular, note that with these hypotheses
the constant $C$ appearing in Lemma~\ref{init_reg} is independent of $\eps$.
Hence, we infer
\begin{gather*}
  \norm{u_\eps}_{W^{1,\infty}(0,T; V^*)\cap H^1(0,T; V)}
  +\norm{v_\eps}_{W^{1,\infty}(0,T; H_\Gamma)\cap H^1(0,T; H^{1/2}(\Gamma))}
  +\eps^{1/2}\norm{v_\eps}_{H^1(0,T; V_\Gamma)}\leq c\,,\\
  \norm{\mu_\eps}_{L^\infty(0,T; V)\cap L^2(0,T; W_{\bf n}\cap H^3(\Omega))}\leq c\,,\\
  \norm{\eta_\eps}_{L^\infty(0,T; H)} + \norm{\eta_{\Gamma\eps}}_{L^\infty(0,T; H_\Gamma)}
  +\norm{\xi_\eps}_{L^\infty(0,T; H)}\leq c\,,\\
  \norm{\Delta u_\eps}_{L^\infty(0,T; H)} + \norm{\partial_{\bf n}u_\eps-\eps\Delta_\Gamma v_\eps + \xi_{\Gamma\eps}}_{L^\infty(0,T; H_\Gamma)}\leq c\,.
\end{gather*}
Now, arguing as before by elliptic regularity and interpolation arguments, we deduce that 
\[
  \norm{\partial_{\bf n}u_\eps}_{L^\infty(0,T; H^{-1/2}(\Gamma))} +
  \eps\norm{\Delta_\Gamma v_\eps}_{L^\infty(0,T; H^{-1/2}(\Gamma))}+\norm{\xi_{\Gamma\eps}}_{L^\infty(0,T; H^{-1/2}(\Gamma))}\leq c\,.
\]
Hence, 
the conclusion of the proof follows easily by a completely similar argument.



\begin{thebibliography}{10}

\bibitem{bcst1}
E.~Bonetti, P.~Colli, L.~Scarpa, and G.~Tomassetti.
\newblock A doubly nonlinear {C}ahn-{H}illiard system with nonlinear viscosity.
\newblock {\em Commun. Pure Appl. Anal.}, 17(3):1001--1022, 2018.

\bibitem{brezis}
H.~Br\'ezis.
\newblock {\em Op\'erateurs maximaux monotones et semi-groupes de contractions
  dans les espaces de {H}ilbert}.
\newblock North-Holland Publishing Co., Amsterdam-London; American Elsevier
  Publishing Co., Inc., New York, 1973.
\newblock North-Holland Mathematics Studies, No. 5. Notas de Matem\'atica (50).

\bibitem{cahn-hill}
J.~W. Cahn and J.~E. Hilliard.
\newblock Free energy of a nonuniform system. i. interfacial free energy.
\newblock {\em The Journal of Chemical Physics}, 28(2):258--267, 1958.

\bibitem{cal-colli}
L.~Calatroni and P.~Colli.
\newblock Global solution to the {A}llen-{C}ahn equation with singular
  potentials and dynamic boundary conditions.
\newblock {\em Nonlinear Anal.}, 79:12--27, 2013.

\bibitem{col-far-hass-gil-spr}
P.~Colli, M.~H. Farshbaf-Shaker, G.~Gilardi, and J.~Sprekels.
\newblock Optimal boundary control of a viscous {C}ahn-{H}illiard system with
  dynamic boundary condition and double obstacle potentials.
\newblock {\em SIAM J. Control Optim.}, 53(4):2696--2721, 2015.

\bibitem{colli-fuk-CHmass}
P.~Colli and T.~Fukao.
\newblock Cahn-{H}illiard equation with dynamic boundary conditions and mass
  constraint on the boundary.
\newblock {\em J. Math. Anal. Appl.}, 429(2):1190--1213, 2015.

\bibitem{col-fuk-eqCH}
P.~Colli and T.~Fukao.
\newblock Equation and dynamic boundary condition of {C}ahn-{H}illiard type
  with singular potentials.
\newblock {\em Nonlinear Anal.}, 127:413--433, 2015.

\bibitem{col-gil-spr}
P.~Colli, G.~Gilardi, and J.~Sprekels.
\newblock On the {C}ahn-{H}illiard equation with dynamic boundary conditions
  and a dominating boundary potential.
\newblock {\em J. Math. Anal. Appl.}, 419(2):972--994, 2014.

\bibitem{col-gil-spr-contr}
P.~Colli, G.~Gilardi, and J.~Sprekels.
\newblock A boundary control problem for the pure {C}ahn-{H}illiard equation
  with dynamic boundary conditions.
\newblock {\em Adv. Nonlinear Anal.}, 4(4):311--325, 2015.

\bibitem{col-gil-spr-contr2}
P.~Colli, G.~Gilardi, and J.~Sprekels.
\newblock A boundary control problem for the viscous {C}ahn-{H}illiard equation
  with dynamic boundary conditions.
\newblock {\em Appl. Math. Optim.}, 73(2):195--225, 2016.

\bibitem{col-scar}
P.~Colli and L.~Scarpa.
\newblock From the viscous {C}ahn-{H}illiard equation to a regularized
  forward-backward parabolic equation.
\newblock {\em Asymptot. Anal.}, 99(3-4):183--205, 2016.

\bibitem{colli-sprek-optACDBC}
P.~Colli and J.~Sprekels.
\newblock Optimal control of an {A}llen-{C}ahn equation with singular
  potentials and dynamic boundary condition.
\newblock {\em SIAM J. Control Optim.}, 53(1):213--234, 2015.

\bibitem{colli-visin}
P.~Colli and A.~Visintin.
\newblock On a class of doubly nonlinear evolution equations.
\newblock {\em Comm. Partial Differential Equations}, 15(5):737--756, 1990.

\bibitem{ef-zel}
M.~Efendiev and S.~Zelik.
\newblock Finite-dimensional attractors and exponential attractors for
  degenerate doubly nonlinear equations.
\newblock {\em Math. Methods Appl. Sci.}, 32(13):1638--1668, 2009.

\bibitem{fish-spinod}
H.~P. Fischer, P.~Maass, and W.~Dieterich.
\newblock Novel surface modes in spinodal decomposition.
\newblock {\em Phys. Rev. Lett.}, 79:893--896, Aug 1997.

\bibitem{gal-DBC2}
C.~G. Gal.
\newblock On a class of degenerate parabolic equations with dynamic boundary
  conditions.
\newblock {\em J. Differential Equations}, 253(1):126--166, 2012.

\bibitem{gal-DBC}
C.~G. Gal.
\newblock The role of surface diffusion in dynamic boundary conditions: {W}here
  do we stand?
\newblock {\em Milan J. Math.}, 83(2):237--278, 2015.

\bibitem{gal-grass-ACDBC}
C.~G. Gal and M.~Grasselli.
\newblock The non-isothermal {A}llen-{C}ahn equation with dynamic boundary
  conditions.
\newblock {\em Discrete Contin. Dyn. Syst.}, 22(4):1009--1040, 2008.

\bibitem{gil-mir-sch}
G.~Gilardi, A.~Miranville, and G.~Schimperna.
\newblock On the {C}ahn-{H}illiard equation with irregular potentials and
  dynamic boundary conditions.
\newblock {\em Commun. Pure Appl. Anal.}, 8(3):881--912, 2009.

\bibitem{gil-mir-sch-longtime}
G.~Gilardi, A.~Miranville, and G.~Schimperna.
\newblock Long time behavior of the {C}ahn-{H}illiard equation with irregular
  potentials and dynamic boundary conditions.
\newblock {\em Chin. Ann. Math. Ser. B}, 31(5):679--712, 2010.

\bibitem{gurtin}
M.~E. Gurtin.
\newblock Generalized {G}inzburg-{L}andau and {C}ahn-{H}illiard equations based
  on a microforce balance.
\newblock {\em Phys. D}, 92(3-4):178--192, 1996.

\bibitem{brezzi-gilardi}
H.~Kardestuncer and D.~H. Norrie, editors.
\newblock {\em {\em Chapters 1--3 in} Finite element handbook}.
\newblock McGraw-Hill Book Co., New York, 1987.

\bibitem{kenz-spinod}
R.~Kenzler, F.~Eurich, P.~Maass, B.~Rinn, J.~Schropp, E.~Bohl, and
  W.~Dieterich.
\newblock Phase separation in confined geometries: {S}olving the
  {C}ahn-{H}illiard equation with generic boundary conditions.
\newblock {\em Computer Physics Communications}, 133(2):139 -- 157, 2001.

\bibitem{maier-stan2}
S.~Maier-Paape and T.~Wanner.
\newblock Spinodal decomposition for the {C}ahn-{H}illiard equation in higher
  dimensions. {I}. {P}robability and wavelength estimate.
\newblock {\em Comm. Math. Phys.}, 195(2):435--464, 1998.

\bibitem{maier-stan1}
S.~Maier-Paape and T.~Wanner.
\newblock Spinodal decomposition for the {C}ahn-{H}illiard equation in higher
  dimensions: nonlinear dynamics.
\newblock {\em Arch. Ration. Mech. Anal.}, 151(3):187--219, 2000.

\bibitem{mir-sch}
A.~Miranville and G.~Schimperna.
\newblock On a doubly nonlinear {C}ahn-{H}illiard-{G}urtin system.
\newblock {\em Discrete Contin. Dyn. Syst. Ser. B}, 14(2):675--697, 2010.

\bibitem{mir-zel}
A.~Miranville and S.~Zelik.
\newblock Robust exponential attractors for {C}ahn-{H}illiard type equations
  with singular potentials.
\newblock {\em Math. Methods Appl. Sci.}, 27(5):545--582, 2004.

\bibitem{mir-zel2}
A.~Miranville and S.~Zelik.
\newblock Doubly nonlinear {C}ahn-{H}illiard-{G}urtin equations.
\newblock {\em Hokkaido Math. J.}, 38(2):315--360, 2009.

\bibitem{novick-cohen}
A.~Novick-Cohen.
\newblock On the viscous {C}ahn-{H}illiard equation.
\newblock In {\em Material instabilities in continuum mechanics ({E}dinburgh,
  1985--1986)}, Oxford Sci. Publ., pages 329--342. Oxford Univ. Press, New
  York, 1988.

\bibitem{sch-seg-stef}
G.~Schimperna, A.~Segatti, and U.~Stefanelli.
\newblock Well-posedness and long-time behavior for a class of doubly nonlinear
  equations.
\newblock {\em Discrete Contin. Dyn. Syst.}, 18(1):15--38, 2007.

\bibitem{simon}
J.~Simon.
\newblock Compact sets in the space {$L^p(0,T;B)$}.
\newblock {\em Ann. Mat. Pura Appl. (4)}, 146:65--96, 1987.

\end{thebibliography}
\end{document}